\documentclass[12pt,oneside]{amsart}
\usepackage{graphics}
\usepackage[dvips]{graphicx}
\usepackage{cancel}
\usepackage{amsmath,amsfonts}
\usepackage[update,prepend]{epstopdf}
\usepackage{tikz}
\usepackage{times}
\usepackage{float}
\usepackage[ansinew]{inputenc}
\usepackage{amssymb}
\usepackage{amsmath}
\usepackage{subfigure}
\usepackage{amsfonts}
\usepackage{latexsym}
\usepackage{epsfig}
\usepackage{geometry}
\usepackage[update,prepend]{epstopdf}
\usepackage[colorlinks = true,
linkcolor = black, citecolor=black]{hyperref}
\def\2010mathclass#1{\par%
	\insert\footins{\footnotesize{\it\textup{2010} Mathematics
			Subject Classification:} #1}}
		
\numberwithin{equation}{section}
\newtheorem{Theorem}{Theorem}[section]

\newtheorem{Lemma}{Lemma}[section]

\newtheorem{assum}{Assumption} [section]

\newtheorem{Remark}{Remark}[section]

\begin{document}
\title[Zaremba boundary condition]{Semi-classical defect measure and internal stabilization for the semilinear wave equation subject to Zaremba
boundary conditions}

\author[Cavalcanti]{Marcelo M. Cavalcanti}
\address{ Department of Mathematics, State University of
Maring\'a, 87020-900, Maring\'a, PR, Brazil.}
\email{mmcavalcanti@uem.br}
\thanks{Research of Marcelo M. Cavalcanti partially supported by the CNPq Grant
300631/2003-0}

\author[Cornilleau]{ Pierre Cornilleau}
\address{  Lyc\'ee Pothier, 2 bis, rue Marcel Proust, 45000 Orl\'eans , France.}
\email{: pierre.cornilleau@ens-lyon.org}
%\thanks{Research of }

\author[Domingos Cavalcanti]{Val\'eria N. Domingos Cavalcanti}
\address{ Department of Mathematics, State University of Maring\'a,
87020-900, Maring\'a, PR, Brazil.}
\email{vndcavalcanti@uem.br}
\thanks{Research of Val\'eria N. Domingos Cavalcanti partially supported by the CNPq Grant
304895/2003-2}

\author[Robbiano]{Luc Robbiano}
\address{Laboratoire de Math\'ematiques de Versalilles, Universit\'e de Versailles
St Quentin, CNRS, 45, Avenue des Etats-Unis, 78035 Versailles, France.}
\email{Luc.Robbiano@uvsq.fr}
%\thanks{Research of }


\2010mathclass{35L05, 35l20, 35B40.}%

\maketitle

\begin{abstract}
In this article we exploite the uniform decay for damped linear wave equation with Zaremba 
boundary condition, obtained in a previous work, to treat the same problem in nonlinear context. 
We need a uniqueness assumption, usual for this type of nonlinear problem. The result is deduced from 
an observation estimate for nonlinear problem proved by a contradiction argument.
\end{abstract}

%%%%%%%%%%%%%%%%%%%%%%%%%%%%%%%%%%%%%
\maketitle

\tableofcontents

\section{Introduction}

\subsection{Description of the Problem.}

This article is devoted to the analysis of the exponential and uniform decay rates
of solutions to the wave equation subject to a localized frictional damping and Zaremba boundary conditions:
\begin{equation}\label{eq:*}
\left\{
\begin{aligned}
&{\partial_t^2 u -\Delta u + f(u) + a(x) \partial_t u = 0\quad \hbox{in}\,\,\,\Omega \times (0, +\infty),}\\\
&{u=0\quad \hbox{on}\quad \partial \Omega_D \times (0,+\infty ),}\\\
&{\partial_\nu u=0\quad \hbox{on}\quad \partial \Omega_N \times (0,+\infty ),}\\\
&{u(x,0)=u_0(x);\quad \partial_t u(x,0)=u_1(x),\quad x\in\Omega,}
\end{aligned}
\right.
\end{equation}
where $\Omega$ is a bounded domain of $\mathbb{R}^n ,$ $n\geq 1$, with smooth boundary $\partial \Omega=\partial \Omega_D \cup \partial \Omega_N $,~ $\partial \Omega_D \cap \partial \Omega_N =\emptyset$, $meas(\partial \Omega_D)\ne 0$, $meas(\partial \Omega_N)\ne 0$, $f:\mathbb{R}\rightarrow \mathbb{R}$ is a $C^2$ function with sub-critical growth which satisfies the sign condition $f(s) s \geq 0$, for all $s\in \mathbb{R}$ (see further assumptions \eqref{ass on f} and \eqref{ass on F}). Here, $M:=(\overline{\Omega}, G)$ is a compact Riemannian manifold where we are inducing on $\Omega$ a Riemannian metric $G$, $\nabla\equiv \nabla_G$ is the associated  Levi-Civita connection and $\Delta$ represents the Laplace Beltrami operator.

The following assumptions are made on the function $a(x)$, responsible for the localized dissipative effect of frictional type:
\begin{assum}\label{ass 1}
We assume that $a(\cdot) \in L^\infty(\Omega)$ is a nonnegative function. In addition, that $\omega$ geometrically controls $\Omega$, i.e there exists $T_{0} >0$, such that every geodesic of the metric $G$, travelling with speed $1$ and issued at $t = 0$, enters the set $\omega$ in a time $t <  T_{0}$.

Furthermore,
\begin{eqnarray}\label{damp term}
a(x) \geq a_0 >0 \hbox{ a. e. in } \omega.
\end{eqnarray}
\end{assum}

Setting $$H^1_{\partial \Omega_D}(\Omega):= \{u \in H^1(\Omega): u=0 \hbox{ on }\partial \Omega_D\}$$ endowed, thanks to Poincar\'e inequality, with its natural topology $$||u||_{H^1_{\partial \Omega_D}(\Omega)}^2:= \int_\Omega |\nabla u|^2 \,dx,$$ let also assume the following unique continuation principle holds:
\begin{assum}\label{UCP}
     For every $T > 0$, the only solution $v$ lying in the space $C(]0,T[; L^2(\Omega)) \cap C(]0,T[,H_{\partial \Omega_D}^{-1}(\Omega))$, to system
     \begin{equation}
     \left\{
     \begin{aligned}
     & \partial_{t}^2 v - \Delta v + V(x,t)v =0\hbox{ in } \Omega\times (0,T),\\
     &v=0\hbox{ on } \omega,
     \end{aligned}
     \right.
     \end{equation}
     where $V(x, t)  \in L^{\infty}(\Omega \times (0,T)))$, is the trivial
one $v \equiv 0$. Here,$H_{\partial \Omega_D}^{-1}(\Omega)=\left[H^1_{\partial \Omega_D}(\Omega)\right]'$.
\end{assum}

\subsection{Previous Results, Main Goal and Methodology.}

The contribution of the present paper is to introduce a new and a more general approach to obtain the exponential stability of problem (\ref{eq:*}), which generalizes the previous results, and, in addition, can be used for other equations as well regardless of the type of dissipation mechanism considered.
In order to obtain the desired stability result for the wave equation subject to a frictional damping, we consider an approximate problem and we show that its solution decays exponentially to zero in the weak phase space. The method of proof combines an observability inequality, microlocal analysis tools and unique continuation properties. Then, passing to the limit, we recover the original model and prove its global existence as well as the exponential stability.
\medskip

In what follows we are going to explain briefly the methodology we are going to use.

\medskip

Setting
\begin{eqnarray*}
 D(-\Delta):=\{v\in H^1_{\partial \Omega_D}(\Omega): \Delta u \in L^2(\Omega) \},
\end{eqnarray*}
and denoting $v=u_{t}$ we may rewrite problem \eqref{eq:*} as the following Cauchy problem in $\mathcal{H}=H^1_{\partial \Omega_D}(\Omega) \times L^2(\Omega)$
\begin{equation}\label{cauchy problem}
\left\{
\begin{aligned}
&\frac{\partial}{\partial t}(u,v) = A(u,v) + \mathcal{F}(u,v) \\
\\
& (u,v)(0)= (u_{0}, v_{0}),
\end{aligned}
\right.
\end{equation}
where the linear unbounded operator $A: D(A) \rightarrow \mathcal{H}$ is given by
\begin{eqnarray}\label{A'}
A(u,v) = (v, \Delta u - a(x) v),
\end{eqnarray}
with domain
\begin{equation}\label{domain}
D(A) = D(-\Delta) \times H^1_{\partial \Omega_D}(\Omega),
\end{equation}
and $\mathcal{F}: \mathcal{H} \rightarrow \mathcal{H}$ is the nonlinear operator
\begin{equation}\label{F}
\mathcal{F}(u,v) = (0, -f(u)).
\end{equation}

 It is well known that the operator $A: D(A) \subset  \mathcal{H} \rightarrow \mathcal{H}$ defined by \eqref{A'} and \eqref{domain} generates a $C_{0}$-semigroup of contractions $e^{At}$ on the energy space $\mathcal{H}$ and $D(A)$ is dense in $\mathcal{H}$. For more details, see \cite{Pazy}.
Thus, given $\{u_0,u_1\} \in H^1_{\partial \Omega_D}(\Omega) \times L^2(\Omega)$, consider a sequence $\{u_{0,k},u_{1,k}\}\in D(A)$, satisfying
\begin{eqnarray}\label{conv init data}
\{u_{0,k},u_{1,k}\} \rightarrow \{u_0,u_1\} ~\hbox{ in } H^1_{\partial \Omega_D}(\Omega) \times L^2(\Omega).
\end{eqnarray}

 Thus, instead of studying problem (\ref{eq:*}) directly, we shall study, for each $k \in \mathbb{N}$, the auxiliary problem
\begin{equation}\label{eq:AP}
\left\{
\begin{aligned}
&{\partial_t^2 u_k -\Delta u_k + f_k(u_k) + a(x) \partial_t u_k= 0\quad \hbox{in}\,\,\,\Omega \times (0, +\infty),}\\\
&{u_k=0\quad \hbox{on}\quad \partial \Omega_D \times (0,+\infty ),}\\\
&{\partial_\nu u_k=0\quad \hbox{on}\quad \partial \Omega_N \times (0,+\infty ),}\\\
&{u_k(x,0)=u_{0,k}(x);\quad \partial_t u_k(x,0)=u_{1,k}(x),\quad x\in\Omega,}
\end{aligned}
\right.
\end{equation}
where $f_k: \mathbb{R} \longrightarrow \mathbb{R}$ is defined by
\begin{equation}\label{trunc func}
f_{k}(s):=
\begin{cases}
f(s),  &  |s|\leq k,\\
f(k),  & s> k,\\
f(-k), & s < -k.
\end{cases}
\end{equation}

Here, we use some ideas from Lasiecka and Tataru's work \cite{Lasiecka-Tataru} adapted to the present context.
The energy identity associated to problem (\ref{eq:AP}) is given by
\begin{eqnarray}\label{Trunc ener Ident}
E_{u_k}(t) + \int_0^t \int_\Omega a(x) |\partial_t u_k(x,s)|^2\,dxds = E_{u_k}(0), \hbox{ for all } t\in [0,+\infty) \hbox{ and } k\in \mathbb{N},
\end{eqnarray}
where
\begin{eqnarray}\label{Trunc ener}
E_{u_k}(t):= \frac12 \int_{\Omega} |\partial_t u_k(x,t)|^2 + |\nabla u_k(x,t)|^2\,dx + \int_\Omega F_k(u_k(x,t))\,dx,
\end{eqnarray}
with $F_k(s):= \int_0^s f_k(\lambda)\,d\lambda$. Furthermore, we will also prove the corresponding observability inequality to problem (\ref{eq:AP}), that is, we shall prove that there exists a positive constant $C$  which does not depend on $k$, verifying
\begin{equation}\label{Truc Obs Ineq}
E_{u_k}(0) \leq C\,\int_0^{T}\int_{\Omega} a(x) |\partial_t u_k|^2 \,dx\,dt,~\hbox{ for all } T\geq T_0.
\end{equation}

 Finally, passing to the limit in (\ref{Trunc ener Ident}) and (\ref{Truc Obs Ineq}) as $k\rightarrow +\infty$, we achieve the energy identity and  the observability inequality associated to problem (\ref{eq:*}), respectively, which are the necessary and sufficient ingredients to establish its exponential stability result. However, in order to established (\ref{Truc Obs Ineq}) we need two facts:~(i)~To prove the observability inequality associated to the linear problem:
 \begin{equation}\label{LP}
\left\{
\begin{aligned}
&{\partial_t^2 y -\Delta y = 0\quad \hbox{in}\,\,\,\Omega \times (0, T),}\\\
&{y=0\quad \hbox{on}\quad \partial \Omega_D \times (0, T),}\\\
&{\partial_\nu y=0\quad \hbox{on}\quad \partial \Omega_N \times (0, T),}\\\
&{y(x,0)=y_0(x)\in H^1_{\partial \Omega_D}(\Omega);\quad \partial_t y(x,0)=y_1(x)\in L^2(\Omega),\quad x\in\Omega,}
\end{aligned}
\right.
\end{equation}
namely, there exists a constant $c>0$ such that
\begin{eqnarray}\label{Linear Obs Ineq}
E_y^L(0) \leq c \int_0^T \int_\omega |\partial_t y(x,t)|^2\,dxdt,
\end{eqnarray}
 for all $(y_0,y_1) \in H^1_{\partial \Omega_D}(\Omega) \times L^2(\Omega)$, where $E_y^L(t):=  \frac12 \int_{\Omega} |\partial_t y(x,t)|^2 + |\nabla y(x,t)|^2\,dx\,dx$. (ii)~The second main ingredient in the proof is to consider the well known property which establishes the linear map $\{z_0, z_1, f\}\in H^1_{\partial_D}(\Omega) \times L^2(\Omega) \times L^1(0,T;L^2(\Omega))\mapsto \{z, \partial_t z\}\in L^\infty(0,T;H^1_{\partial \Omega_D}(\Omega)) \times L^\infty(0,T;L^2(\Omega))$ associated to problem
 \begin{equation*}
\left\{
\begin{aligned}
&{\partial_t^2 z -\Delta z = f\quad \hbox{in}\,\,\,\Omega \times (0, T),}\\\
&{z=0\quad \hbox{on}\quad \partial \Omega_D \times (0, T),}\\\
&{\partial_\nu z=0\quad \hbox{on}\quad \partial \Omega_N \times (0, T),}\\\
&{z(x,0)=z_0(x)\in H^1_{\partial \Omega_D}(\Omega);\quad \partial_t z(x,0)=z_1(x)\in L^2(\Omega),\quad x\in\Omega,}
\end{aligned}
\right.
\end{equation*}
is continuous, that is,
 \begin{eqnarray}\label{continuos map}
 &&||z||_{L^\infty(0,T;H^1_{\partial \Omega_D}(\Omega))}^2 + ||\partial_t z||_{L^\infty(0,T;L^2(\Omega))}^2\\
 &&\lesssim ||z_0||_{ H^1_{\partial \Omega_D}(\Omega)}^2 + ||z_1||_{L^\infty(0,T;L^2(\Omega))}^2 + ||f||_{L^1(0,T;L^2(\Omega))}^2.\nonumber
 \end{eqnarray}

 It is worth mentioning, according proved by Haraux \cite{Haraux}, the equivalence between the exponential decay of solutions to the second order evolution equation:
 \begin{equation}\label{DLP}
\left\{
\begin{aligned}
&{\partial_t^2 y -\Delta y + a(x) \partial_t y= 0\quad \hbox{in}\,\,\,\Omega \times (0, T),}\\\
&{y=0\quad \hbox{on}\quad \partial \Omega_D \times (0, T),}\\\
&{\partial_\nu y=0\quad \hbox{on}\quad \partial \Omega_N \times (0, T),}\\\
&{y(x,0)=y_0(x)\in H^1_{\partial \Omega_D}(\Omega);\quad \partial_t y(x,0)=y_1(x)\in L^2(\Omega),\quad x\in\Omega,}
\end{aligned}
\right.
\end{equation}
(uniformly on bounded sets of $\mathcal{H}$), and the `controllability property' given in (\ref{Linear Obs Ineq}) of the system governed by the undamped equation (\ref{LP}). As a consequence, instead of proving (\ref{Linear Obs Ineq}) it is sufficient to prove the exponential decay of weak solutions to problem (\ref{DLP}). In order to do that, refined arguments of microlocal analysis will be considered jointly with the characterization given by \cite{Huang} (Theorem 3), namely:
\begin{Theorem}[Gearhart–Pr\"uss–Huang]~ Let $e^{At}$ be a $C_0$-semigroup in a
Hilbert space $H$ and assume that there exists a positive constant $M >0$ such that
$|||e^{At}||| \leq M$ for all $t \geq 0$. Then $e^{At}$ is exponentially stable if and only if $i\mathbb{R} \subset \rho(A)$
and
\begin{eqnarray}\label{RE}
\sup_{\mu \in \mathbb{R}} ||| \left(A-i\mu I_d \right)^{-1}|||_{\mathcal{L(H)}}<+\infty.
\end{eqnarray}
\end{Theorem}

 \section{Convergence of the auxiliary problem}
\setcounter{equation}{0}

\medskip
\subsection{The limit process.}
\medskip

In this section we prove that the sequence $\{u_k\}_{k\in \mathbb{N}}$ of solutions to problem (\ref{eq:AP}) converges to the unique solution to the problem (\ref{eq:*}).

The function $f$ satisfies the following hypotheses:
\begin{assum}\label{assumption 2.1}
$f:\mathbb{R}\rightarrow \mathbb{R}$ is a $C^2$ function with sub-critical growth; satisfying the sign condition $f(s) s \geq 0$, for all $s\in \mathbb{R}$, and
\begin{eqnarray}\label{main ass on f}
f(0)=0, \ \ \ |f^{(j)}(s)| \leq k_{0}(1 + |s|)^{p-j}, \hbox{ for all } s \in \mathbb{R} \hbox{ and } j=1,2.
\end{eqnarray}
In particular, we obtain from (\ref{main ass on f}),
\begin{eqnarray}\label{ass on f}
|f(r) - f(s)| \leq c \left(1 + |s|^{p-1} + |r|^{p-1} \right)|r-s|,~\hbox{ for all } s,r\in \mathbb{R},
\end{eqnarray}
for some $c>0$, with
\begin{eqnarray}\label{ass on f'}
1\leq p \leq \frac{n+2}{n-2}~\hbox{ if }n\geq 3~\hbox{ or } ~p\geq 1~\hbox{ if }~ n=1,2.
\end{eqnarray}

In addition,
\begin{eqnarray}\label{ass on F}
0 \leq F(s) \leq f(s) s,~\hbox{ for all } s\in \mathbb{R},
\end{eqnarray}
where $F(\lambda):= \int_0^\lambda f(s)\,ds$.
\medskip
\end{assum}

We begin with some preliminary results.
\begin{Lemma}\label{Lemma1}
The distributional derivative $f_k'$ of the function defined in (\ref{trunc func}) is the essentially bounded function $g_k:\mathbb{R} \rightarrow \mathbb{R}$ given by
\begin{equation}\label{drrivative}
g_k(s):=
\begin{cases}
f'(s), & |s|\leq k,\\
0, & s> k,\\
0, & s < -k.
\end{cases}
\end{equation}
\end{Lemma}
\begin{proof}
Take $\varphi \in C_0^\infty(\mathbb{R})$. Once $f_k\in L^1_{loc}(\mathbb{R})$ we have
\begin{eqnarray*}
&&\left<f_k',\varphi \right>_{\mathcal{D}'(\mathbb{R}), \mathcal{D}(\mathbb{R})}= - \int_{\mathbb{R}} f_k(s) \varphi'(s)\,ds\\
&&= -\left[\int_{-\infty}^{-k} f_k(s) \varphi'(s)\,ds + \int_{-k}^k f_k(s) \varphi'(s)\,ds  + \int_{k}^{+\infty} f_k(s) \varphi'(s)\,ds  \right]\\
&&= - \left[ f(-k) \varphi(-k)  + f(k) \varphi(k) - f(-k) \varphi(-k) - \int_{-k}^k f'(s) \varphi(s)\,ds - f(k)\varphi(k)\right]\\
&& =  \int_{-k}^k f'(s) \varphi(s)\,ds = \int_{\mathbb{R}} g(s) \varphi(s) \,ds.
\end{eqnarray*}
\end{proof}

Consider the following result which will be useful to the proof of Lemma \ref{Lema2}.
\begin{Theorem}\label{brezis}
	Let $u \in W^{1,p}(I)$ with $1 \leq p \leq \infty$, where $I$ is a bounded interval of $\mathbb{R}$. Then, there exists $\widetilde{u} \in C(\bar{I})$ such that
	\begin{equation*}
	u=\widetilde{u} \hbox{ a.e. in } I
	\end{equation*}
	and
	\begin{equation*}
	\widetilde{u}(x)-\widetilde{u}(y)=\int_{y}^{x} u'(t)dt \hbox{ for all } x, y \in \bar{I}.
	\end{equation*}
\end{Theorem}
\begin{proof}
See Brezis \cite{Brezis}, Theorem 8.2.
\end{proof}
\begin{Lemma}\label{Lema2}
For each $k\in \mathbb{N}$, there exists a positive constant $C_k$ verifying
\begin{equation*}
|f_k(r)-f_k(s)|\leq C_k|r-s| \hbox{ for every} r,s \in  \mathbb{R},
\end{equation*}
where $f_k$ is the function defined in \eqref{trunc func}.
\end{Lemma}
\begin{proof}
Consider $s, r\in \mathbb{R}$ with $s< r$. Applying Theorem \ref{brezis} for $I=]s,r[$,  it follows that
\begin{eqnarray*}
f_k(r) - f_k(s) = \int_s^r f_k'(\xi)\,d\xi.
\end{eqnarray*}
Thus, Lemma \ref{Lemma1} yields the following inequality:
\begin{eqnarray}\label{Lipschitz}
|f_k(r) - f_k(s)| \leq \int_s^r |f_k'(\xi)|\,d\xi \leq \sup_{s\in [-k,k]}|g_k(s)|\,|r-s|,
\end{eqnarray}
which concludes the proof.
\end{proof}

From Lemma \ref{Lema2}, for each $k\in \mathbb{N}$, standard arguments of Semigroup theory yield that problem (\ref{eq:AP}) possesses an unique regular solution $u_k$ in the class
\begin{eqnarray*}
C^0([0,\infty); D(-\Delta)) \cap C^1([0,\infty); H^1_{\partial \Omega_D}(\Omega)) \cap C^2([0,\infty); L^2(\Omega)).
\end{eqnarray*}

Multiplying the first equation of (\ref{eq:AP}) by $\partial_t u_k$ and performing integration by parts, it yields
\begin{eqnarray}\label{est1}
&&\frac12 \frac{d}{dt} ||\partial_t u_k(t)||_{L^2(\Omega)}^2 + \frac12 \frac{d}{dt} ||\nabla u_k(t)||_{L^2(\Omega)}^2 + \frac{d}{dt}\int_\Omega F_k(u_k(x,t))\,dxdt\\
&& + \int_\Omega a(x) |\partial_t u_k(x,t)|^2\,dx =0, \hbox{ for all } t\in [0,\infty),\nonumber
\end{eqnarray}
where
\begin{eqnarray}\label{primitive}
F_k(\lambda) = \int_0^\lambda f_k(s)\,ds.
\end{eqnarray}

Hence, taking (\ref{est1}) into account, we infer
\begin{eqnarray}\label{est2}
E_{u_k}(t) &+& \int_0^t \int_\Omega a(x) |\partial_t u_k(x,s)|^2\,dxds = E_{u_k}(0), \hbox{ for all } t\in [0,+\infty) \hbox{ and } k\in \mathbb{N},\nonumber
\end{eqnarray}
where
\begin{eqnarray}
E_{u_k}(t):= \frac12 \int_\Omega  |\partial_t u_k(x,t)|^2 + |\nabla u_k(x,t)|^2\,dx + \int_\Omega F_k(u_k(x,t))\,dx,
\end{eqnarray}
is the energy associated to problem (\ref{eq:AP}).

We observe that from (\ref{trunc func}), the function defined in (\ref{primitive}) is given by
\begin{equation}\label{primitive trunc func}
F_{k}(s):=\begin{cases}
\displaystyle \int_0^sf(\xi)\,d\xi, & |s|\leq k,\\
\displaystyle \int_0^k f(\xi) \,d\xi + f(k)[s-k], & s > k,\\
\displaystyle f(-k)[s+k] + \int_0^{-k}f(\xi)\,d\xi, & s < -k.
\end{cases}
\end{equation}

Since $f$ satisfies the sign condition, it results that $F_k(s) \geq 0$ for all $s\in \mathbb{R}$ and $k\in \mathbb{N}$. In addition, from (\ref{ass on f}) and (\ref{ass on F}), we obtain, respectively, that $|f(s)|\leq c [|s| + |s|^p]$ and $0 \leq F(s) \leq f(s) \,s$ for all $s\in \mathbb{R}$.  Then, we infer that
\begin{eqnarray}\label{bound on Fk}
|F_k(s)| \leq c [|s|^2 + |s|^{p+1}],~\hbox{ for all }s\in \mathbb{R}~\hbox{ and }~k\in\mathbb{N}.
\end{eqnarray}

Consequently,
\begin{eqnarray}\label{bound data L1}
\int_\Omega |F_k(u_{0,k})|\,dx &\leq& c\int_\Omega \left[|u_{0,k}|^2 + |u_{0,k}|^{p+1} \right]\,dx\\
&\lesssim&  ||u_{0,k}||_{H^1_{\partial \Omega_D}(\Omega)}.\nonumber
\end{eqnarray}

Assuming that $p\geq 1$ is under conditions (\ref{ass on f'}), we have for every dimension $n\geq 1$ that $H^1_{\partial \Omega_D}(\Omega) \hookrightarrow L^{p+1}(\Omega)$, which implies that the RHS of  (\ref{bound data L1}) is bounded.
So, estimates (\ref{est2}) (also called energy identity for the auxiliary problem (\ref{eq:AP}) and (\ref{bound data L1}) and convergence (\ref{conv init data}),  yield a subsequence of $\{u_k\}$, reindexed again by $\{u_k\}$,  such that
\begin{eqnarray}
&&u_k \rightharpoonup u ~\hbox{ weakly * in } L^\infty(0,\infty; H^1_{\partial \Omega_D}(\Omega)),\label{conver1}\\
&&\partial_t u_k \rightharpoonup \partial_t u ~\hbox{ weakly * in } L^\infty(0,\infty; L^2(\Omega))\label{conver2},\\
&& \sqrt{a(x)} \partial_t u_{k} \rightharpoonup \sqrt{a(x)} \partial_t u \hbox{ weakly in }L^{2}(0,\infty; L^{2}(\Omega)).\label{eq:3.56}
\end{eqnarray}

Employing the standard compactness result (see Simon \cite{Simon}) we also deduce that
\begin{eqnarray}\label{conver3}
u_k \rightarrow u~\hbox{ strongly in } L^\infty (0,T; L^{2^{\ast}-\eta}(\Omega)); \hbox{ for all } T>0,
\end{eqnarray}
where $2^{\ast}:= \frac{2n}{n-2}$ and $\eta >0$ is small enough. In addition, from (\ref{conver3}), we obtain
\begin{eqnarray}\label{conver4}
u_k \rightarrow u~\hbox{ a. e. in  } \Omega \times (0,T), \hbox{ for all } T>0.
\end{eqnarray}

On the other hand, from (\ref{ass on f}), (\ref{ass on f'}), (\ref{conver1}) and once $H^1_{\partial \Omega_D}(\Omega)\hookrightarrow  L^{p+1}(\Omega)\hookrightarrow  L^{\frac{p+1}{p}}(\Omega)$
the following estimate holds:
\begin{eqnarray}\label{estIII}
\|f_k(u_k)\|_{L^{\frac{p+1}{p}}}^{\frac{p+1}{p}} &=& \int_0^{T} \int_\Omega |f_k(u_k(x,t))|^{\frac{p+1}{p}}\,dxdt \nonumber\\
&\lesssim& \int_0^T\int_\Omega |u_k|^{\frac{p+1}{p}}\,dxdt + \int_{0}^{T}\int_\Omega |u_k|^{p+1}\,dx dt\nonumber\\
&=& \int_0^T \|u_k\|_{L^{\frac{p+1}{p}}(\Omega) }^{\frac{p+1}{p}}\,dt+\int_0^T \|u_k\|_{L^{p+1}(\Omega)}^{p+1}\, dt\nonumber\\
&\lesssim& \int_0^T \|u_k\|_{H^1_{\partial \Omega_D}(\Omega)}^{\frac{p+1}{p}}\,dt+\int_0^T \|u_k\|_{H^1_{\partial \Omega_D}(\Omega)}^{p+1}\, dt\nonumber\\
&\lesssim& \|u_k\|_{L^\infty(0,T;H^1_{\partial \Omega_D}(\Omega))}^{\frac{p+1}{p}}+\|u_k\|_{L^\infty(0,T;H^1_{\partial \Omega_D}(\Omega))}^{p+1}\nonumber\\
& \leq& c <+\infty,~\hbox{ for all }t\geq 0.
\end{eqnarray}
It is easy to see that
\begin{equation}\label{estIII.1}
f(u) \in L^\infty(0,\infty;L^{\frac{p+1}{p}}(\Omega)).
\end{equation}
Indeed,
\begin{eqnarray}\label{estIII.2}
\int_\Omega |f(u(x,t))|^{\frac{p+1}{p}}\,dx &\lesssim& \int_\Omega |u(x,t)|^{\frac{p+1}{p}}\,dx + \int_\Omega |u(x,t)|^{p+1}\,dx \nonumber\\
& \lesssim& \|u(\cdot,t)\|_{H^1_{\partial \Omega_D}(\Omega)}^{\frac{p+1}{p}}+ \|u(\cdot,t)\|_{H^1_{\partial \Omega_D}(\Omega)}^{p+1}\nonumber\\
& \lesssim& \|u\|_{L^\infty(0,T;H^1_{\partial \Omega_D}(\Omega))}^{\frac{p+1}{p}} +\|u\|_{L^\infty(0,T;H^1_{\partial \Omega_D}(\Omega))}^{p+1}<+\infty,~\hbox{ for all }t\geq 0.
\end{eqnarray}
From \eqref{estIII.2} and the definition of essential supremum we obtain \eqref{estIII.1}.

In addition, from (\ref{conver4}) and the continuity of the function $f$, we get
\begin{eqnarray}\label{conver5''}
f_k(u_k) \rightarrow f(u) \hbox{ a. e. in  } \Omega \times (0,T), \hbox{ for all } T>0.
\end{eqnarray}

Indeed, the convergence (\ref{conver4}) guarantees the existence of set $Z_T \subset \Omega \times (0,T)$ with $\operatorname{meas}(Z_T)=0$  such that $u_k(x,t) \rightarrow u(x,t)$ for all $(x,t) \in \Omega \times (0,T) \setminus Z_T$ when $k \rightarrow \infty$. Therefore, for all $(x,t) \in \Omega \times (0,T) \setminus Z_T$ there exists a positive constant $L=L(x,t)>0$ verifying $|u_k(x,t)|<L,$ for all $k \in \mathbb{N}$.  Then, using the definition of $f_k$, we obtain that
\begin{eqnarray}\label{boundedness}
\hbox{ if } |u_k(x,t)|<L, \hbox{ for all } k \in \mathbb{N} \hbox{ then }~ f_k(u_k(x,t))=f(u_k(x,t)), \hbox{ for all } k\geq L,
\end{eqnarray}
that is,
\begin{equation}\label{boundednes1}
f_k(u_k(x,t))-f(u_k(x,t)) \rightarrow 0 \hbox{ when } k \rightarrow \infty \hbox{ for all } (x,t) \in \Omega \times (0,T) \setminus Z_T.
\end{equation}
On the other hand, employing the continuity of $f$ it follows that
\begin{equation}\label{boundednes2}
f(u_k(x,t))-f(u(x,t)) \rightarrow 0 \hbox{ when } k \rightarrow \infty \hbox{ for all } (x,t) \in \Omega \times (0,T) \setminus Z_T.
\end{equation}
From \eqref{boundednes1} and \eqref{boundednes2} the convergence (\ref{conver5''}) holds.

\begin{Lemma}[Strauss] Let $\mathcal{O}$ be an open and bounded subset of $\mathbb{R}^N$, $N\geq 1$, $1<q<+\infty$ and $\{u_n\}_{n\in \mathbb{N}}$ a sequence which is bounded in $L^q(\mathcal{O})$. If $u_n \rightarrow u$ a.e. in $\mathcal{O}$, then $u\in L^q(\mathcal{O})$ and $u_n \rightharpoonup u$ weakly in $L^q(\mathcal{O})$. In addition, if $1\leq r <q$ we also have $u_n \rightarrow u$ strongly in $L^r(\mathcal{O})$.
\end{Lemma}
\begin{proof}
	See  \cite{Brezis}~(Exercise 4.16) or \cite{Strauss}.
\end{proof}

Gathering together \eqref{estIII}, \eqref{estIII.1} and Lions' Lemma, we deduce that \begin{eqnarray}\label{weak conv fk}
f_k(u_k) \rightharpoonup f(u)~\hbox{ weakly in } L^{\frac{p+1}{p}}(\Omega \times (0,T)).
\end{eqnarray}

Going back to problem (\ref{eq:AP}), multiplying by $ \varphi \, \theta $, where $ \varphi \in C_0^\infty(\Omega), \theta \in C_0^\infty(0,T)$ and performing integration by parts, we obtain
 \begin{eqnarray}\label{form var}
&&-\int_0^T \theta'(t) \int_\Omega \partial_t u_{k}(x,t)\, \varphi(x) \,dx dt + \int_0^T \theta(t) \int_\Omega \nabla u_{k}(x,t)\cdot \nabla \varphi(x) \,dxdt\\
&& +\int_0^T \theta(t) \int_\Omega f_k(u_k(x,t))\,\varphi(x) \,dxdt +\int_0^T \theta(t) \int_\Omega a(x) \partial_t u_k(x,t) \varphi(x)\,dxdt=0.\nonumber
 \end{eqnarray}

Passing to the limit in (\ref{form var}) and observing convergences (\ref{conver1})-(\ref{eq:3.56}) and (\ref{weak conv fk}), we get
 \begin{eqnarray}\label{limit}
&&-\int_0^T \theta'(t) \int_\Omega \partial_t u(x,t)\, \varphi(x) \,dx dt + \int_0^T \theta(t) \int_\Omega \nabla u(x,t)\cdot \nabla \varphi(x) \,dxdt\\
&& +\int_0^T \theta(t) \int_\Omega f(u(x,t))\,\varphi(x) \,dxdt + \int_0^T \theta(t) \int_\Omega a(x) \partial_t u(x,t) \varphi(x)\,dxdt\nonumber = 0,
 \end{eqnarray}
 for all $\varphi \in C_0^\infty(\Omega)$ and $\theta \in C_0^\infty(0,T)$.  We conclude that
 \begin{eqnarray}\label{dist sol}
 \partial_t^2 u - \Delta u + f(u) +a(x)\partial_t u = 0 ~\hbox{ in }\mathcal{D}'(\Omega \times (0,T)),
 \end{eqnarray}
 and since
 \begin{eqnarray*}
&& a(\cdot) \partial_t u \in L^\infty(0,T; L^2(\Omega)), ~\Delta u \in L^\infty(0,T; H^{-1}_{\partial \Omega_D}(\Omega)),~(\hbox{here } H^{-1}_{\partial \Omega_D}(\Omega) = (H^1_{\partial_{\Omega_D}}(\Omega))'), \\
&& a(x) \partial_t u \in L^2(0,T; L^2(\Omega)) ~\hbox{and}~ f(u) \in L^\infty(0,T; L^{\frac{p+1}{p}}(\Omega)),
 \end{eqnarray*}
we deduce that $\partial_t^2 u \in  L^2(0,T;  H^{-1}_{\partial \Omega_D}(\Omega))$ and
\begin{eqnarray}\label{weak solution'}
 \partial_t^2 u - \Delta u + f(u) + a(x) \partial_t u = 0 ~\hbox{ in } L^2(0,T;   H^{-1}_{\partial \Omega_D}(\Omega)).
\end{eqnarray}

Applying Lemma 8.1 of Lions-Magenes \cite{Lions-Magenes}, we deduce that
\begin{equation}\label{weak solution}
u\in C_w(0,T;H^1_{\partial \Omega_D}(\Omega)) \hbox{ and } \partial_t u \in C_w(0,T; L^2(\Omega)),
\end{equation}
where $C_w(0,T;Y)=$ space of functions $f\in L^\infty(0,T;Y)$ whose mappings $[0,T] \mapsto Y$ are weakly continuous, that is, $t \mapsto \langle y', f(t) \rangle_{Y',Y}$ is continuous in $[0,T]$ for all $y' \in Y'$, dual of $Y$.

\medskip

Our first result reads as follows:
 \begin{Theorem}\label{theo 1}
 Assume that $a\in L^\infty(\Omega)$  and $f\in C^1(\mathbb{R})$ satisfies $f(s)s\geq 0$ for all $s\in \mathbb{R}$.  In addition, suppose that assumptions (\ref{ass on f}), (\ref{ass on f'}) and (\ref{ass on F}) are in place.  Then, problem (\ref{eq:*}) has at least a global solution in the class
 {\small
 $$u\in C_w(0,T;H^1_{\partial \Omega_D}(\Omega)),~\partial_t u \in C_w(0,T; L^2(\Omega)),~\partial_t^2 u \in L^2(0,T; H_{\partial \Omega_D}^{-1}(\Omega)),$$}
 provided that $\{u_0,u_1\}\in H^1_{\partial \Omega_D}(\Omega) \times L^2(\Omega)$. Furthermore, assuming that $1\leq p \leq \frac{n}{n-2},n\geq 3$ or $p\geq 1, n=1,2$, we have the uniqueness of solution.
 \end{Theorem}

\medskip
 \begin{proof}
The uniqueness of solution as well as to prove that $u(0)=u_0$ and $\partial_t u(0)=u_1$ follow the same ideas used in Lions \cite{Lions1} (Theorem 1.2).
\end{proof}

\medskip

\subsection{Recovering the regularity in time for the range $1\leq p < \frac{n}{n-2},$ $n\geq 3$. }

When $p\geq 1,$ $n=1,2$, the result is trivially verified and it will be omitted.

The goal of this subsection is to prove that if $1\leq p < \frac{n}{n-2},$ $n\geq 3$, the related solutions to problem (\ref{eq:*}) are in the class
$$u\in C^0([0,T];H^1_{\Omega_D}(\Omega)),~\partial_t u \in C^0([0,T]; L^2(\Omega))$$
and, in addition, one has
\begin{eqnarray*}
\{u_k, \partial_t u_k\} \rightarrow \{u,\partial_tu\} \hbox{ in } C^0([0,T];H^1_{\Omega_D}(\Omega)) \times C^0([0,T]; L^2(\Omega)).
\end{eqnarray*}

To prove the above statements, we need to prove that
\begin{eqnarray}\label{main conv}
f_k(u_k) \rightarrow f(u) ~\hbox{ strongly in }L^2(\Omega \times (0,T)).
\end{eqnarray}

In fact, first we observe that
\begin{eqnarray}\label{I}
&&\int_0^T \int_\Omega |f_k(u_k) - f(u)|^2\,dxdt\\
&& \lesssim \int_0^T \int_\Omega |f_k(u_k) - f(u_k)|^2\,dxdt + \int_0^T \int_\Omega |f(u_k) - f(u)|^2\,dxdt.\nonumber
\end{eqnarray}

In view of (\ref{ass on f}) one has
\begin{eqnarray*}
\int_\Omega |f(u_k) - f(u)|^2 \,dx &\lesssim& \int_\Omega |u_k - u|^2\,dx + \int_\Omega |u_k|^{2(p-1)} |u_k -u|^2\,dx \nonumber \\
&&+ \int_\Omega |u|^{2(p-1)} |u_k -u|^2\,dx \nonumber \\
&=& I_{1,k}+I_{2,k}+I_{3,k}
\end{eqnarray*}

We observe that since $\frac{p-1}{p} + \frac{1}{p}=1$, H\"older inequality yields
\begin{eqnarray*}
I_{2,k} \leq \left(\int_\Omega |u_k|^{2p} \right)^{\frac{p-1}{p}}\left( \int_\Omega |u_k - u|^{2p}\right)^{\frac{1}{p}}.
\end{eqnarray*}

Choosing $p< \frac{n}{n-2}$ it implies that $2p < \frac{2n}{n-2}=2^{\ast}$ and, consequently, from (\ref{conver1}) and (\ref{conver3}) we deduce that $I_{2,k} \rightarrow 0$ as $k\rightarrow +\infty$. Analogously, we also deduce that $I_{3,k} \rightarrow 0$ as $k\rightarrow +\infty$.  We trivially obtain that $I_{1,k} \rightarrow 0$ as $k\rightarrow +\infty$.  Then,
\begin{eqnarray}\label{II}
\int_0^T \int_\Omega |f(u_k) - f(u)|^2\,dxdt \rightarrow 0 ~\hbox{ as } k \rightarrow \infty.
\end{eqnarray}

From (\ref{I}) it remains to prove that
\begin{eqnarray}\label{III}
\int_0^T \int_\Omega |f_k(u_k) - f(u_k)|^2 \,dxdt \rightarrow 0 ~\hbox{ as } k \rightarrow \infty.
\end{eqnarray}

Let us consider, initially, $t\in [0,T]$ fixed and define $$\Omega_k^t:=\{x\in \Omega: |u_k(x,t)| >k\}.$$

Observing that
\begin{eqnarray*}
f_k(u_k) - f(u_k) =0, ~\hbox{ if } |u_k(x,t)|\leq k,
\end{eqnarray*}
we have

\begin{eqnarray}\label{IV}
&&\int_\Omega |f_k(u_k) - f(u_k)|^2\,dx = \int_{\Omega_k^t} |f_k(u_k) - f(u_k)|^2\,dx\\
&&\lesssim \left[\int_{\Omega_k^t} |f(u_k)|^2\, dx + \int_{\Omega_k^t} |f(-k)|^2\, dx + \int_{\Omega_k^t} |f(k)|^2\, dx  \right]\nonumber\\
&& \lesssim\left[ \int_{\Omega_k^t} [|u_k|^2 + |u_k|^{2p} ]\, dx +  \int_{\Omega_k^t} [|k|^2 + |k|^{2p} ]\, dx\right]\nonumber\\
&& \lesssim\left[ \int_{\Omega_k^t} |u_k|^{2p} \, dx +  \int_{\Omega_k^t} |k|^{2p}\, dx\right]\nonumber\\
&& \lesssim \int_{\Omega_k^t} |u_k|^{2p} \,dx.\nonumber
\end{eqnarray}

Before analyzing the term on the RHS of (\ref{IV}) we note that since $H^1_{\Omega_D}(\Omega)\hookrightarrow L^{\frac{2n-\frac12}{n-2}}(\Omega)$ and the convergence (\ref{conv init data}) are in place, we obtain
\begin{eqnarray} \label{V}
\left( \int_{\Omega_k^t} k^{\frac{2n-\frac{1}{2}}{n-2}}\,dx \right) &\lesssim& \left( \int_{\Omega_k^t}
|u_k|^{\frac{2n-\frac{1}{2}}{n-2}}\,dx \right)\\
%&=&  \left(||u_k(t)||_{L^{\frac{2n-\frac12}{n-2}}(\Omega)}^{\frac{2n-\frac12}{n-2}}\right)\nonumber\\
&=& ||u_k(t)||_{L^{\frac{2n-\frac12}{n-2}}(\Omega_k^t)}^{\frac{2n-\frac12}{n-2}} \lesssim ||u_k(t)||_{H^1_{\Omega_D}(\Omega)}^{\frac{2n-\frac12}{n-2}} \lesssim [E_{u_k}(0)]^{\frac{2n-\frac12}{n-2}} \leq C,\nonumber
\end{eqnarray}
for all $t\in [0,T]$, where $C$ is a positive constant which does not depend on $k$ and $t$.
Thus, it yields
\begin{eqnarray}\label{VI}
\operatorname{meas}(\Omega_k^t)  \lesssim k^{\frac{-2n+\frac12}{n-2}}, \hbox{ for all } t\in [0,T].
\end{eqnarray}

\medskip

Let $\beta:= \frac{2n}{(2p)(n-2)}$, for $n\geq 3$. Observe that we have the following inequalities:
$$p< \frac{n}{n-2} \Leftrightarrow  2n> (2p)(n-2)\Leftrightarrow 2p< \frac{2n}{n-2}=2^{\ast} \Leftrightarrow  \beta > 1. $$

Setting $\alpha>0$ such that $\frac{1}{\alpha} + \frac{1}{\beta}=1$, we deduce that $\alpha=\frac{2n}{2n-(2p)(n-2)}$ and using H\"older inequality we get
\begin{eqnarray}\label{VII'}
\int_{\Omega_k^t} |u_k|^{2p}\,dx &\leq& \left( \operatorname{meas}(\Omega_k)\right)^{\frac{2n-(2p)(n-2)}{2n}} \left(\int_{\Omega_k^t} |u_k|^{\frac{2n}{n-2}}\right)^{\frac{(2p)(n-2)}{2n}}\\
&=& \left( \operatorname{meas}(\Omega_k)\right)^{\frac{2n-(2p)(n-2)}{2n}} ||u_k(t)||_{L^{\frac{2n}{n-2}}(\Omega)}^{2p}.\nonumber
\end{eqnarray}

Thus, from (\ref{VI}) and (\ref{VII'}) we conclude
\begin{eqnarray}\label{VIII'}
\int_0^T\int_{\Omega_k^t} |u_k|^{2p}\,dx &\leq& k^{\left(\frac{-2n+\frac12}{n-2}\right)\left(\frac{2n-(2p)(n-2)}{2n}\right)} \int_0^T||u_k(t)||_{L^{\frac{2n}{n-2}}(\Omega)}^{2p}\,dt\\
&\lesssim& k^{\left(\frac{-2n+\frac12}{n-2}\right)\left(\frac{2n-(2p)(n-2)}{2n}\right)}\int_0^T ||u_k(t)||_{H^1_{\Omega_D}(\Omega)}^{2p}\,dt\nonumber\\
&\lesssim& k^{\left(\frac{-2n+\frac12}{n-2}\right)\left(\frac{2n-(2p)(n-2)}{2n}\right)}[ E_{u_k}(0)]^p,\nonumber
\end{eqnarray}

Employing the fact that $E_{u_k}(0) \leq C$ for all $k\in \mathbb{N}$ and $\left(\frac{-2n+\frac12}{n-2}\right)\left(\frac{2n-(2p)(n-2)}{2n}\right)< 0$, in light of inequality (\ref{VIII'}) , we prove that

\begin{eqnarray}\label{VIII'.convergente}
\int_0^T\int_{\Omega_k^t} |u_k|^{2p}\,dx \rightarrow 0,\, \hbox{ as } k\rightarrow +\infty.
\end{eqnarray}

Gathering (\ref{IV}) and (\ref{VIII'.convergente}) together, we conclude (\ref{III}) which proves (\ref{main conv}).

\medskip

Now, we define the sequence $z_{\mu,\sigma}=u_{\mu}-u_{\sigma}$, $\mu,\sigma\in
\mathbb{N}$, and from (\ref{eq:AP}) we deduce
\begin{eqnarray}\label{eq:3.59}
&&\frac12 \frac{d}{dt} \left\{||\partial_t z_{\mu,\sigma}(t)||_{L^2(\Omega)}^2 +||\nabla
z_{\mu,\sigma}(t)||_{L^2(\Omega)}^2 \right\}
+\int_{\Omega} a(x)|\partial_t z_{\mu,\sigma}|^2\,dx \\
&&=  \int_{\Omega}\left(f_{\mu}(u_{\mu})-
f_{\sigma}(u_{\sigma})\right)(\partial_t u_{\mu}-\partial_t u_{\sigma}) \,dx.
\nonumber
\end{eqnarray}

Integrating (\ref{eq:3.59}) over $(0,t)$, we obtain
\begin{eqnarray}\label{eq:3.60}
&&\frac12  \left\{||\partial_t z_{\mu,\sigma}(t)||_{L^2(\Omega)}^2 +||\nabla z_{\mu,\sigma}(t)||_{L^2(\Omega)}^2 \right\}
+\int_0^t\int_{\Omega} a(x)|\partial_t z_{\mu,\sigma}|^2\,dx ds\\
&& \leq  \frac12\left\{||u_{1,\mu} - u_{1,\sigma}||_{L^2(\Omega)}^2 + ||\nabla u_{0,\mu}
- \nabla u_{0,\sigma}||_{L^2(\Omega)}^2\right\}\nonumber\\
&&+ \int_0^t \int_{\Omega}\left(f_{\mu}(u_{\mu})-
f_{\sigma}(u_{\sigma})\right)(\partial_t u_{\mu}-\partial_t u_{\sigma}) \,dxds.
\nonumber
\end{eqnarray}

 The convergences (\ref{conv init data}), (\ref{conver2}) and (\ref{main conv}) imply that the terms on the RHS of  the  (\ref{eq:3.60}) converges to zero
as $\mu, \sigma \rightarrow +\infty$.  Thus, we deduce that
\begin{eqnarray}\label{Cauchy conv}\quad
&&u_{\mu} \rightarrow u \hbox { in }
C^0([0,T];H^1_{\Omega_D}(\Omega)) \cap C^1 ([0,T]; L^2(\Omega)),\\
&&\lim_{\mu \rightarrow +\infty}
\int_0^T\int_{\Omega} a(x) | \partial_t u_{\mu}|^2
dx\, ds = \int_0^T\int_{\Omega} a(x) |\partial_t u|^2
dx\, ds,\label{damping conv'}
\end{eqnarray}
for all $T>0$.

\medskip

\subsection{Estimating $F_k(u_k)$ }

\medskip

Inequality (\ref{bound on Fk}) gives $$|F_k(s)| \leq c [|s|^2 + |s|^{p+1}],$$ for all $s\in \mathbb{R}$ and $k\in \mathbb{N}$.

Since $1\leq p < \frac{n}{n-2}$ if $n\geq 3$ and $ \frac{n}{n-2} <  \frac{n+2}{n-2}$ we obtain $2\leq p+1 < \frac{2n}{n-2}=2^{\ast}$. Consequently, there exists $\varepsilon >0$ such that $p+1+\varepsilon=2^{\ast}$.  Then, $H^1_{\Omega_D}(\Omega) \hookrightarrow L^{p+1+\varepsilon}(\Omega)$ and, consequently,
\begin{eqnarray}\label{bound data Lp'}
\int_\Omega |F_k(u_{0,k})|^\frac{p+1+\varepsilon}{p+1}\,dx &\leq& c\int_\Omega |u_{0,k}|^{\frac{2(p+1+\varepsilon)}{p+1}} + |u_{0,k}|^{p+1+\varepsilon} \,dx\\
&\lesssim&  ||u_{0,k}||_{H^1_{\Omega_D}(\Omega)}^{p+1+\varepsilon} \leq C.\nonumber
\end{eqnarray}
Analogously,
\begin{eqnarray}\label{bound on Fk'}
\int_\Omega |F_k(u_{k}(x,t_0))|^\frac{p+1+\varepsilon}{p+1}\,dx
\lesssim ||u_{k}(\cdot,t_0)||_{H^1_{\Omega_D}(\Omega)}^{p+1+\varepsilon} \leq C E_{u_k}(0)^{p+1+\varepsilon},
\end{eqnarray}
for all $t_0\in [0,T]$.  The boundedness of $E_{u_k}(0)$ implies that there exists $\chi \in L^{\frac{2^{\ast}}{p+1}}(\Omega )$ verifying the following convergence:
\begin{eqnarray}\label{weak conv of Fk}
F_k(u_k(\cdot,t_0))\rightharpoonup \chi ~\hbox{ weakly in }L^{\frac{2^{\ast}}{p+1}}(\Omega ),~\hbox{ as } k \rightarrow +\infty.
\end{eqnarray}

%%%%%%%%%%%%%%%%%%%%%%%%%%%%%%%%%%%%%%%%%%%%%%%%%%%%%%%%%%%%%%%%%%%%%%%%%%%%%%%%%%%%%%%%%%%%%%%%%%%%%%%%%%%%%%%%%%%%%%%%%%%%%%%%%%%%%

%%%%%%%%%%%%%%%%%%%%%%%%%%%%%%%%%%%%%%%%%%%%%%%%%%%%%%%%%%%%%%%%%%%%%%%%%%%%%%%%%%%%%%%%%%%%%%%%%%%%%%%%%%%%%%%%%%%%%%%%%%%%%%%%%%%%%%%%%%%

In what follows we are going to prove that $\chi=F(u(\cdot,t_0))$. Indeed, from \eqref{Cauchy conv} we obtain $u_k(\cdot,t_0) \rightarrow u(\cdot,t_0)$ strongly in $L^2(\Omega)$.  Thus,
\begin{equation}\label{victor1}
u_k(x,t_0) \rightarrow u(x,t_0) \hbox{ a. e. in } \Omega.
\end{equation}

Note that,
\begin{eqnarray}\label{victor3}
&&|F_k(u_k(x,t_0))-F(u(x,t_0))| \\
&&\leq |F_k(u_k(x,t_0))-F(u_k(x,t_0))|+|F(u_k(x,t_0))-F(u(x,t_0))|.\nonumber
\end{eqnarray}

The convergence  (\ref{victor1}) and the continuity of $F$ imply
\begin{eqnarray}\label{victor4}
F(u_k(x,t_0)) \rightarrow F(u(x,t_0)) \hbox{ a. e. in } \Omega.
\end{eqnarray}

In light of inequality (\ref{victor3}), to prove that
\begin{eqnarray}\label{conv a.e. F_k}
F_k(u_k(x,t_0)) \rightarrow F(u(x,t_0)) \hbox{ a. e. in } \Omega.
\end{eqnarray}
it remains to prove that
\begin{eqnarray*}
F_k(u_k(x,t_0)) - F(u_k(x,t_0)) \rightarrow 0 \hbox{ a. e. in } \Omega,
\end{eqnarray*}

In fact, from (\ref{boundedness}), there exists a positive constant $L=L(x,t)>0$ such that \begin{eqnarray}\label{victor5}
|F_k(u_k(x,t_0))-F(u_k(x,t_0))| &=& \left|\int_{0}^{u_k(x,t_0)} f_k(s)ds-\int_{0}^{u_k(x,t_0)} f(s)ds\right|\nonumber\\
&\leq& \int_{-L}^{L} |f_k(s)-f(s)|ds =0,~\hbox{ if }k\geq L.
\end{eqnarray}

Therefore, combining \eqref{victor3}, \eqref{victor4} and \eqref{victor5}, we obtain (\ref{conv a.e. F_k}). Thus, from (\ref{bound on Fk'}) and Lions Lemma we deduce that
\begin{eqnarray}\label{main weak conv}
F_k(u_k(\cdot,t_0))\rightharpoonup F(u(\cdot,t_0)) ~\hbox{ weakly in }L^{\frac{2^{\ast}}{p+1}}(\Omega), \hbox{ as } k \rightarrow +\infty,
\end{eqnarray}
proving that $\chi=F(u(\cdot,t_0))$.

In addition, employing Strauss Lemma we also deduce that
\begin{eqnarray}\label{main strong conv}
F_k(u_k(\cdot, t_0))\rightarrow F(u(\cdot,t_0)) ~\hbox{ strongly in }L^r(\Omega),~\hbox{ as } k \rightarrow +\infty,
\end{eqnarray}
 for all $1\leq r < \frac{2^{\ast}}{p+1}$ and $t_0 \in [0,T]$.

\medskip

Now we are in a position to establish the following result:
\begin{Theorem}\label{theo 2}
 Assume that $a\in L^\infty(\Omega)$ is a nonnegative function and $f\in C^1(\mathbb{R})$ satisfies $f(s)s\geq 0$ for all $s\in \mathbb{R}$.  In addition, suppose that $f$ verifies assumption (\ref{ass on f}) with $1\leq p < \frac{n}{n-2},n\geq 3$ and $p\geq 1, n=1,2$ and assumption (\ref{ass on F}).  Then, given $\{u_0,u_1\}\in H^1_{\Omega_D}(\Omega) \times L^2(\Omega)$ problem (\ref{eq:*}) has an unique global solution in the class
 {\small
 $$u\in C^0([0,T];H^1_{\Omega_D}(\Omega)),~\partial_t u \in C^0([0,T]; L^2(\Omega)),~\partial_t^2 u \in L^2(0,T; H^{-1}_{\partial \Omega_D}(\Omega) ).$$}

In addition, the energy identity is verified, namely
\begin{eqnarray}\label{energyidentity}\qquad
E_u(t_2) + \int_{t_1}^{t_2} \int_\Omega a(x) |a(x) \partial_t u(x,t)|^2\,dxdt=E_u(t_1),~0\leq t_1 \leq t_2<+\infty,~\hbox{ where }\\
E_{u}(t):= \frac12\int_\Omega |\partial_t u(x,t)|^2 + |\nabla u(x,t)|^2\,dxdt + \int_\Omega F(u(x,t))\,dxdt.\nonumber
\end{eqnarray}
 \end{Theorem}

\section{Exponential Decay to Problem (\ref{eq:*})}

Throughout this section we will assume that $1\leq p < \frac{n}{n-2}$ if $n\geq 3$ and $p\geq 1$ if $n=1,2$.  Under these conditions we have the following embeddings:
\begin{eqnarray}\label{tower}
H^1_{\Omega_D}(\Omega) \hookrightarrow L^{2p}(\Omega)  \hookrightarrow L^{p}(\Omega).
\end{eqnarray}
Consider the auxiliary problem
\begin{equation}\label{Aux Prob}
\left\{
\begin{aligned}
&{\partial_t^2 u_k -\Delta u_k + f_k(u_k) + a(x)\partial_t u_k = 0\quad \hbox{in}\,\,\,\Omega \times (0, +\infty),}\\\
&{u_k=0\quad \hbox{on}\quad \partial \Omega_D \times (0,+\infty ),}\\\
&{\partial_\nu u_k=0\quad \hbox{on}\quad \partial \Omega_N \times (0,+\infty ),}\\\
&{u_k(x,0)=u_{0,k}(x);\quad \partial_t u_k(x,0)=u_{1,k}(x),\quad x\in\Omega,}
\end{aligned}
\right.
\end{equation}
whose associated energy functional is given by
\begin{eqnarray}\label{energy}\qquad
E_{u_k}(t):= \frac12\int_\Omega |\partial_t u_k(x,t)|^2 + |\nabla u_k(x,t)|^2\,dxdt + \int_\Omega F_k(u_k(x,t))\,dxdt,
\end{eqnarray}
where  $F_k(\lambda) = \int_0^\lambda F_k(s) \,ds$ and the energy identity reads as follows
\begin{eqnarray}\label{ident energy mu}
E_{u_k}(t_2) - E_{u_k}(t_1) = - \int_{t_1}^{t_2}\int_\Omega  a(x) |\partial_t u_k|^2 \,dxdt,
\end{eqnarray}
for all $0 \leq t_1 \leq t_2 <+\infty$.

Let $T_0>0$ be associated to the geometric control condition, that is, every ray of the geometric optics enters $\omega$ in a time $T^*<T_0$. Thus, our goal is to prove the observability inequality established in the following lemma.

\begin{Lemma} ~ There exists $k_0 \geq 1$ such that for every $k\geq k_0$, the corresponding solution  $u_k$ of \eqref{Aux Prob} satisfies the inequality
\begin{eqnarray}\label{obs ineq}
E_{u_k}(0) \leq C \int_{0}^{T}\int_\Omega  a(x) |\partial_t u_k|^2 \,dxdt \,dxdt,
\end{eqnarray}
for all $T>T_0$ and for some positive constant $C=C(||\{u_0,u_1\}||_{H_0^1(\Omega)\times L^2(\Omega)})$.
\end{Lemma}
\begin{proof}
The initial datum $\{u_0,u_1\}\in H^1_{\Omega_D}(\Omega)\times L^2(\Omega)$ in the original problem (\ref{eq:*}) is either zero or not zero.

In the first case, when $\{u_0,u_1\} =(0,0)$ and, observing (\ref{conv init data}), we can consider  $\{u_{0,k},u_{1,k}\} =(0,0)$ for all $k\geq 1$ and the corresponding unique solution to the auxiliary problem (\ref{eq:AP}) will be $u_k \equiv 0$.  Then, (\ref{obs ineq}) is verified.

In the second case, there exists a positive number $R>0$ such that
$$0< ||\{u_0,u_1\}||_{H^1_{\Omega_D}(\Omega)\times L^2(\Omega)} <R,$$
consider, for instance $R=2||\{u_0,u_1\}||_{H^1_{\Omega_D}(\Omega)\times L^2(\Omega)}$.

Therefore, there exists, $k_0 \geq 1$ such that for all $k \geq k_0$, $\{u_{0,k},u_{1,k}\}$ satisfies
\begin{eqnarray}\label{sequencebound}
||\{u_{0,k},u_{1,k}\}||_{H^1_{\Omega_D}(\Omega)\times L^2(\Omega)} <R.
\end{eqnarray}

We are going to prove that under condition (\ref{sequencebound}) on the initial datum, the corresponding solution $u_k$ to (\ref{eq:AP}) satisfies (\ref{obs ineq}).  Our proof relies on contradiction arguments.  So, if (\ref{obs ineq}) is false, then there exists $T>T_0$ such that for every $k\geq 1$ and every constant $C>0$, there exists an initial datum $\{u^C_{0,k},u^C_{1,k}\}$ verifying (\ref{sequencebound}), whose corresponding solution $u^C_k$ violates (\ref{obs ineq}).

In particular, for every $k\geq 1$ and $C=m\in\mathbb{N}$, we obtain the existence of an initial datum $\{u^m_{0,k},u^m_{1,k}\}$ verifying (\ref{sequencebound}) and whose corresponding solution $u^m_k$ satisfies

\begin{eqnarray}\label{false}
E_{u^m_k}(0)> m \int_{0}^{T}\int_\Omega  a(x) |\partial_t u^m_k|^2 \,dxdt.
\end{eqnarray}

Then, we obtain a sequence $\{u_k^m\}_{m\in \mathbb{N}}$ of solutions to problem (\ref{eq:AP}) such that
\begin{eqnarray*}
\lim_{m \rightarrow +\infty} \frac{E_{u_k^m}(0)}{\int_{0}^{T}\int_\Omega a(x) | \partial_t u_k^m|^2\,dxdt}=+\infty.
\end{eqnarray*}

Equivalently
\begin{eqnarray}\label{normal conv}
\lim_{m \rightarrow +\infty}\frac{\int_{0}^{T}\int_\Omega a(x) |\partial_t u_k^m|^2 \,dxdt}{E_{u_k^m}(0)}=0.
\end{eqnarray}

Since $E_{u_k^m}(0)$ is bounded,  (\ref{normal conv}) yields
\begin{eqnarray}\label{conv damp}
\lim_{m \rightarrow +\infty}\int_{0}^{T}\int_\Omega a(x) |\partial_t u_k^m|^2 \,dxdt=0.
\end{eqnarray}

Furthermore, there exists a subsequence of $\{u_k^m\}_{m\in \mathbb{N}}$, still denoted by $\{u_k^m\}$ , verifying the following convergences:
\begin{eqnarray}
&&u_k^m \rightharpoonup u_k \hbox{ weakly-star in } L^{\infty}(0,T;H^1_{\Omega_D}(\Omega)),~\hbox{ as }m\rightarrow +\infty,\label{conv1'}\\
&&\partial_t u_k^m \rightharpoonup \partial_t u_k \hbox{ weakly-star in } L^{\infty}(0,T; L^2(\Omega)),~\hbox{ as }m\rightarrow +\infty,\label{conv2'}\\
&& u_k^m \rightarrow u_k \hbox{ strongly in } L^{\infty} (0,T; L^q(\Omega)), ~\hbox{ as }m\rightarrow +\infty, ~\hbox{ for all } q\in \left[2, \frac{2n}{n-2}\right),\label{conv3'}
\end{eqnarray}
where the last convergence is obtained using Aubin-Lions-Simon Theorem (see \cite{Simon}).  The proof is divided into two distinguished cases: $u_k\ne 0$ and $u_k=0$.

\medskip
Case~(a):~$u_k\ne 0$. ~
\medskip

For $m\in \mathbb{N}$, $u_k^m$ is the solution to the problem
\begin{equation*}
\left\{
\begin{aligned}
&{\partial_t^2 u_k^m -\Delta u_k^m + f_k(u_k^m) + a(x) \partial_t u_k^m = 0\quad \hbox{in}\,\,\,\Omega \times (0, T),}\\\
&{u_k^m=0\quad \hbox{on}\quad \partial \Omega_D \times (0,T),}\\\
&{\partial_\nu u_k^m=0\quad \hbox{on}\quad \partial \Omega_N \times (0,T),}\\\
&{u_k^m(x,0)=u_{0,k}^m(x);\quad \partial_t u_k^m(x,0)=u_{1,k}^m(x),\quad x\in\Omega.}
\end{aligned}
\right.
\end{equation*}

Taking (\ref{conv damp})-(\ref{conv3'}) into consideration we obtain
\begin{equation}\label{limit1'}
\left\{
\begin{aligned}
&{\partial_t^2 u_k -\Delta u_k + f_k(u_k)= 0\quad \hbox{in}\,\,\,\Omega \times (0, T),}\\\
&{u_k=0\quad \hbox{on}\quad \partial \Omega_D \times (0, T ),}\\\
&{\partial_\nu u_k=0\quad \hbox{on}\quad \partial \Omega_N \times (0, T ),}\\\
&{\partial_t u_k}=0 \hbox{ a.e. in }\omega.
\end{aligned}
\right.
\end{equation}

Defining $y_k = \partial_t u_k$, the above problem yields
\begin{equation*}
\left\{
\begin{aligned}
&{\partial_t^2 y_k -\Delta y_k + f_k'(u_k)y_k= 0\quad \hbox{in}\,\,\,\Omega \times (0, +\infty),}\\\
&{y_k=0\quad \hbox{on}\quad \partial \Omega_D \times (0,+\infty ),}\\\
&{\partial_\nu y_k=0\quad \hbox{on}\quad \partial \Omega_N \times (0,+\infty ),}\\\
&{y_k}=0 \hbox{ a.e. in }\omega.
\end{aligned}
\right.
\end{equation*}

Once $ f_k'(u_k)\in L^\infty (\Omega \times (0,T))$ since $f_k$ is globally Lipschitz, for each $k \in m\in \mathbb{N}$, we deduce from Assumption \ref{UCP} that $y_k= \partial_t u_k\equiv 0$.  Returning to (\ref{limit1'}) we conclude that $u_k\equiv 0$ as well and we obtain the desired contradiction.

\medskip
Case~(b):~$u_k= 0$. ~
\medskip

Setting
\begin{eqnarray}\label{def v_k}
\alpha_m := \sqrt{E_{u_k^m}(0)}, ~ \hbox{ and }~v_k^m:= \frac{u_k^m}{\alpha_m},
\end{eqnarray}
in light of (\ref{normal conv}), we obtain
\begin{eqnarray}\label{damping conv}
\lim_{m \rightarrow +\infty} \int_{0}^{T}\int_\Omega a(x) |\partial_t v_k^m|^2\,dxdt=0.
\end{eqnarray}

According to (\ref{def v_k}), the sequence  $\{v_k^m\}_{m\in \mathbb{N}}$ is the solution to the following problem:
\begin{equation}\label{eq:NP}
\left\{
\begin{aligned}
&{\partial_t^2 v_k^m -\Delta v_k^m + \frac{1}{\alpha_m} f_k(u_k^m)  + a(x) \partial_t v_k^m = 0\quad \hbox{in}\,\,\,\Omega \times (0, T),}\\\
&{v_k^m=0\quad \hbox{on}\quad \partial \Omega_D \times (0,T),}\\\
&{\partial_\nu v_k^m=0\quad \hbox{on}\quad \partial \Omega_N \times (0,T),}\\\
&{v_k^m(x,0)=\frac{u_{0,k}^m}{\alpha_m};\quad \partial_t v_k^m(x,0)=\frac{u_{1,k}^m}{\alpha_m}}
\end{aligned}
\right.
\end{equation}
and the associated energy functional is given by
\begin{eqnarray*}
E_{v_k^m}(t) = \frac12 \int_\Omega \left(|\partial_t v_k^m|^2 + |\nabla v_k^m|^2 \right)\,dx + \frac{1}{\alpha_m^2} \int_\Omega F_k(u_k^m)\,dx,
\end{eqnarray*}
since
$$
\frac{1}{\alpha_m}\int_\Omega f_k(u_k^m) \partial_t v_k^m\,dx =\frac{1}{\alpha_m^2} \frac{d}{dt}\int_\Omega F(u_k^m)\,dx.$$

Note that $E_{v_k^m}(t)= \frac{1}{\alpha_m^2} E_{u_k^m}(t)$ for all $t\geq 0$ and, in particular, for $t=0$
\begin{eqnarray}\label{norm initial energy}
E_{v_k^m}(0)= \frac{1}{\alpha_m^2} E_{u_k^m}(0)=1,~\hbox{ for all }m\in \mathbb{N}.
\end{eqnarray}

In order to achieve the contradiction we are going to prove that
\begin{eqnarray}\label{main goal}
\lim_{m\rightarrow +\infty} E_{v_k^m}(0)=0.
\end{eqnarray}

Indeed, initially, we observe that (\ref{norm initial energy}) yields the existence of a subsequence of $\{v_k^m\}_{m\in \mathbb{N}}$, reindexed again by $\{v_k^m\}$, such that
\begin{eqnarray}
&&v_k^m \rightharpoonup v_k \hbox{ weakly-star in } L^{\infty}(0,T; H^1_{\Omega_D}(\Omega)),~\hbox{ as }m\rightarrow +\infty,\label{conv1}\\
&&\partial_t v_k^m \rightharpoonup \partial_t v_k \hbox{ weakly-star in } L^{\infty}(0,T; L^2(\Omega)),~\hbox{ as }m\rightarrow +\infty,\label{conv2}\\
&& v_k^m \rightarrow v_k \hbox{ strongly in } L^{\infty} (0,T; L^q(\Omega)), ~\hbox{ as }m\rightarrow +\infty, ~\hbox{ for all } q\in \left[2, \frac{2n}{n-2}\right).\label{conv3}
\end{eqnarray}

For some eventual subsequence, we have that $\alpha_m \rightarrow \alpha$ with $\alpha\geq0$.

\medskip

If $\alpha>0$, thus, passing to the limit in (\ref{eq:NP}) and considering convergences (\ref{damping conv}), (\ref{conv1}) - (\ref{conv3}), we deduce
\begin{equation}\label{limit1}
\left\{
\begin{aligned}
&{\partial_t^2 v_k -\Delta v_{k} + \frac{1}{\alpha} f_k(u_k) = 0\quad \hbox{ in }\,\,\,\Omega \times (0, T),}\\\
&{v_k=0\quad \hbox{on}\quad \partial \Omega_D \times (0,T),}\\\
&{\partial_\nu v_k=0\quad \hbox{on}\quad \partial \Omega_N \times (0,T),}\\\
&{\partial_t v_k=0 \hbox{ a.e. in }\omega .}
\end{aligned}
\right.
\end{equation}

The above problem yields, for $w_k= \partial_t v_k$, in the distributional sense,
\begin{equation}
\left\{
\begin{aligned}
&{\partial_t^2 w_k -\Delta w_k + \frac{1}{\alpha} f_k'(u_k) w_k = 0\quad \hbox{ in }\,\,\,\Omega \times (0, T),}\\\
&{w_k=0\quad \hbox{on}\quad \partial \Omega_D \times (0, T),}\\\
&{\partial_\nu w_k=0\quad \hbox{on}\quad \partial \Omega_N \times (0, T),}\\\
&{w_k=0 \hbox{ a.e. in }\omega.}
\end{aligned}
\right.
\end{equation}

Once $\frac{1}{\alpha} f_k'(u_k)\in L^\infty (\Omega \times (0,T))$, using again Assumption (\ref{UCP})  we conclude that $w_k= \partial_t v_k\equiv 0$, and, therefore, returning to (\ref{limit1}) we deduce that $v_k\equiv 0$.

\medskip

If $\alpha=0$, first, observe that hypothesis (\ref{Lipschitz}) yields
\begin{eqnarray*}
\frac{1}{\alpha_m^2} |f_k(u_k^m)|^2 \leq c \frac{1}{\alpha_m^2} |u_k^m|^2 =c \frac{1}{\alpha_m^2}\alpha_m^2 |v_k^m|^2,
\end{eqnarray*}
and
\begin{eqnarray}\label{crucial bound'}
\frac{1}{\alpha_m^2} \int_0^T \int_\Omega |f_k(u_k^m)|^2 \,dxdt \leq  c \int_0^T\int_\Omega |v_k^m|^2\,dxdt.
\end{eqnarray}

We are going to prove that
\begin{equation}\label{nl1}
\frac{1}{\alpha_m}f_k(\alpha_m v_k^m) \rightharpoonup f'(0)v_k \hbox{ in } L^2(0,T;L^2(\Omega)) \hbox{ as } m \rightarrow \infty.
\end{equation}

Since
\begin{equation*}
\frac{1}{\alpha_m}f_k(\alpha_m v_k^m) - f'(0)v_k=\frac{1}{\alpha_m}f_k(\alpha_m v_k^m)-\frac{1}{\alpha_m}f(\alpha_m v_k^m)+\frac{1}{\alpha_m}f(\alpha_m v_k^m) - f'(0)v_k,
\end{equation*}
if we prove that
\begin{equation}\label{nl2}
\frac{1}{\alpha_m}f_k(\alpha_m v_k^m)-\frac{1}{\alpha_m}f(\alpha_m v_k^m) \rightarrow 0 \hbox{ in } L^2(0,T;L^2(\Omega))
\end{equation}
and
\begin{equation}\label{nl3}
\frac{1}{\alpha_m}f(\alpha_m v_k^m) - f'(0)v_k \rightharpoonup 0
\hbox{ in } L^2(0,T;L^2(\Omega)),
\end{equation}
as $m \rightarrow \infty$, we prove (\ref{nl1}).

To prove (\ref{nl2}), let's consider $$\Omega_m^t=\{x \in \Omega: |u_k^m(x,t)|>k\}.$$

Employing definition (\ref{trunc func}), $|f_k(\alpha_m v_k^m)- f(\alpha_m v_k^m)|=0$ in $\Omega \setminus \Omega_m^t$.  Then, hypotheses (\ref{main ass on f}) and (\ref{ass on f}) yield
\begin{small}
\begin{align*}
&\left\| \frac{1}{\alpha_m}f_k(\alpha_m v_k^m)-\frac{1}{\alpha_m}f(\alpha_m v_k^m) \right\|_{L^2(0,T;L^2(\Omega))}^2 \\
= {}& \int_{0}^{T}\int_{\Omega_m^t} \left|\frac{1}{\alpha_m}f_k(\alpha_m v_k^m)-\frac{1}{\alpha_m}f(\alpha_m v_k^m) \right|^2 \, dxdt\\
= {}&\frac{1}{\alpha_m^2}\int_{0}^{T}\int_{\Omega_m^t} \left|f_k(\alpha_m v_k^m)-f(\alpha_m v_k^m) \right|^2 \, dxdt\\
\lesssim {}& \frac{1}{\alpha_m^2}\int_{0}^{T}\int_{\Omega_m^t}|f_k(\alpha_m v_k^m )|^2\, dx dt+\frac{1}{\alpha_m^2}\int_{0}^{T}\int_{\Omega_m^t} |f(\alpha_m v_k^m)|^2 \, dx dt \\
\lesssim {}& \int_{0}^{T}\int_{\Omega_m^t} |f(k)|^2+|f(-k)|^2 \, dx dt+\frac{1}{\alpha_m^2}\int_{0}^{T}\int_{\Omega_m^t} |\alpha_m v_k^m|^2 + |\alpha_m v_k^m|^{2p} \, dx dt\\
\lesssim {}& \int_{0}^{T}\int_{\Omega_m^t} |k|^2+|k|^{2p} \, dx dt+\frac{1}{\alpha_m^2}\int_{0}^{T}\int_{\Omega_m^t} |\alpha_m v_k^m|^2 + |\alpha_m v_k^m|^{2p} \, dx dt\\
\lesssim {}& \frac{1}{\alpha_m^2}\int_{0}^{T}\int_{\Omega_m^t} |u_k^m|^2 + |u_k^m|^{2p} \, dx dt+\frac{1}{\alpha_m^2}\int_{0}^{T}\int_{\Omega_m^t} |\alpha_m v_k^m|^2 + |\alpha_m v_k^m|^{2p} \, dx dt.
\end{align*}
\end{small}

Since $p>1$, $k\geq 1$ and $ k <|u_k^m|=|\alpha_m v_k^m|$ in $\Omega_m^t$, we obtain
\begin{small}
\begin{align*}
\left\| \frac{1}{\alpha_m}f_k(\alpha_m v_k^m)-\frac{1}{\alpha_m}f(\alpha_m v_k^m) \right\|_{L^2(0,T;L^2(\Omega))}^2
\lesssim {}& \frac{1}{\alpha_m^2}\int_{0}^{T}\int_{\Omega_m^t} |\alpha_m v_k^m|^{2p} \, dx dt\\
\lesssim {}& \alpha_m^{2(p-1)} \|v_k^m\|_{L^{2p}(0,T;L^{2p}(\Omega))}^{2p}\rightarrow 0, \hbox{ as } m \rightarrow \infty,
\end{align*}
\end{small}
which proves the convergence \eqref{nl2}.

On the other hand, $f\in C^2(\mathbb{R})$ and, consequently, from Taylor's Theorem and (\ref{main ass on f}) we have
\begin{eqnarray}\label{Taylor}
f(s) = f'(0) s + R(s), \hbox{ where } |R(s)|\leq C(|s|^2 + |s|^p).
\end{eqnarray}

Hence
\begin{eqnarray}\label{ident}
&&\frac{1}{\alpha_m}f(\alpha_m v_k^m)= f'(0) v_k^m + \frac{R(\alpha_m v_k^m)}{\alpha_m}\label{T1}
\end{eqnarray}
and
\begin{equation}\label{T2}
 \left|\frac{R(\alpha_m v_k^m)}{\alpha_m}\right| \leq C \left(\alpha_m|v_k^m|^2 + |\alpha_m|^{p-1} |v_k^m|^p\right).
\end{equation}

In light of identity (\ref{Taylor}), we establish $ \frac{R(\alpha_m v_k^m)}{\alpha_m}=\frac{f(\alpha_m v_k^m)}{\alpha_m}-f'(0)v_k^m$ and hypotheses (\ref{main ass on f}) and (\ref{ass on f}) imply that  $|f(\alpha_m v_k^m)| \lesssim |\alpha_m v_k^m|+|\alpha_m v_k^m|^p$.  Then, we deduce that
\begin{equation*}
\left\| \frac{R(\alpha_m v_k^m)}{\alpha_m} \right\|_{L^2(0,T;L^2(\Omega))}^2 \lesssim \|v_k^m\|_{L^2(0,T;L^2(\Omega))}^2+|\alpha_m|^{2(p-1)}\|v_k^m\|_{L^{2p}(0,T;L^{2p}(\Omega))}^{2p} \leq C,
\end{equation*}
for some constant $C>0$.   We obtain a subsequence of $\frac{R(\alpha_m v_k^m)}{\alpha_m}$ and $\gamma \in L^2(0,T;L^2(\Omega))$ such that
\begin{equation}\label{resto1}
\frac{R(\alpha_m v_k^m)}{\alpha_m} \rightharpoonup \gamma \hbox{ in } L^2(0,T;L^2(\Omega)).
\end{equation}

Besides, employing inequality (\ref{T2}) and observing (\ref{tower}), we get
\begin{eqnarray}\label{resto2}
\left|\left| \frac{R(\alpha_m v_k^m)}{\alpha_m}\right|\right|_{L^1(0,T;L^1(\Omega))}&\lesssim& \int_0^T \int_\Omega \alpha_m |v_k^m|^2 \,dxdt + \int_0^T\int_\Omega \alpha_m^{p-1} |v_k^m|^{p}\,dxdt \nonumber\\
&=& \alpha_m\int_{0}^{T}\|v_k^m\|_{L^2(\Omega)}^2 \, dt+\alpha_m^{p-1}\int_{0}^{T}\|v_k^m\|_{L^p(\Omega)}^{p}\, dt \nonumber \\
&=&  \alpha_m ||v_k^m||_{L^2(0,T; L^2(\Omega))}^2 + \alpha_m^{p-1}||v_k^m||_{L^p(0,T;L^{p}(\Omega))}^{p}\rightarrow 0.
\end{eqnarray}

From \eqref{resto1} and \eqref{resto2} we conclude that
\begin{equation}\label{resto3}
\frac{R(\alpha_m v_k^m)}{\alpha_m} \rightharpoonup 0 \hbox{ in } L^2(0,T;L^2(\Omega)).
\end{equation}

Observing \eqref{conv3}, \eqref{ident} and \eqref{resto3}, the convergence \eqref{nl3} is proved.

\begin{Remark}
	The case $p=1$ is trivially contemplated once the truncation is not necessary.
\end{Remark}

Since convergences (\ref{nl2}) and (\ref{nl3}) are proved, we conclude convergence (\ref{nl1}).

Passing to the limit in (\ref{eq:NP}) as $m\rightarrow +\infty$, we obtain
\begin{equation}\label{limit2}
\left\{
\begin{aligned}
&{\partial_t^2 v_k -\Delta v_k + f'(0) v_k  = 0\quad \hbox{ in }\,\,\,\Omega \times (0, T),}\\\
&{v_k=0\quad \hbox{on}\quad \partial \Omega_D \times (0, T),}\\\
&{v_k=0\quad \hbox{on}\quad \partial \Omega_N \times (0, T),}\\\
&{\partial_t v_k=0 \hbox{ a.e. in }\omega,}
\end{aligned}
\right.
\end{equation}
and defining $w_k= \partial_t v_k$, it satisfies the following problem:
\begin{equation}
\left\{
\begin{aligned}
&{\partial_t^2 w_k -\Delta w_k + f'(0) w_k = 0\quad \hbox{ in }\,\,\,\Omega \times (0, T),}\\\
&{w_k=0\quad \hbox{on}\quad \partial \Omega_D \times (0, T),}\\\
&{w_k=0\quad \hbox{on}\quad \partial \Omega_N \times (0, T),}\\\
&{w_k=0 \hbox{ a.e. in }\omega.}
\end{aligned}
\right.
\end{equation}

Using Assumption (\ref{UCP}) we obtain that $w_k= \partial_t v_k\equiv 0$ and returning to \eqref{limit2} we deduce that $v_k\equiv 0$.

Then, in both cases $\alpha =0$ and $\alpha >0$, we obtain that $v_k\equiv 0$.  Consequently,  inequality (\ref{crucial bound'}) and convergence (\ref{conv3}) yield that
\begin{eqnarray}\label{crucial bound}
\frac{1}{\alpha_m^2} \int_0^T \int_\Omega |f_k(u_k^m)|^2 \,dxdt \rightarrow 0 \hbox{ in } L^2(0,T,L^2(\Omega))~\hbox{ as }m\rightarrow +\infty.
\end{eqnarray}

In order to achieve a contradiction we need to prove that $ E_{v_k^m}(0)\rightarrow 0$ as $m \rightarrow +\infty$.  In fact, from (\ref{eq:NP}), we can write $v_k^m = y_k^m + z_k^m$ such that $y_k^m$ and $v_k^m$ are, respectively, solutions of the following problems:
\begin{equation*}
\left\{
\begin{aligned}
& \partial_t y_k^m - \Delta y_k^m = 0~\hbox{ in }~\Omega \times (0,T),\\
& y_k^m=0 ~\hbox{ on } \partial \Omega_D \times (0,T),\\
& \partial_\nu y_k^m=0 ~\hbox{ on } \partial \Omega_N \times (0,T),\\
& y_k^m(0)=v_k^m(0),~\partial_t y_k^m(0)=\partial_t v_k^m(0),
\end{aligned}
\right.
\end{equation*}
and
\begin{equation*}
\left\{
\begin{aligned}
& \partial_t z_k^m - \Delta z_k^m =  - \frac{1}{\alpha_m} f_k(u_k^m) + a(x) \partial_t u_k^m ~\hbox{ in }~\Omega \times (0,T),\\
& z_k^m=0 ~\hbox{ on } \partial \Omega_D \times (0,T),\\
& \partial_\nu z_k^m=0 ~\hbox{ on } \partial \Omega_N \times (0,T),\\
& z_k^m(0)=0,~\partial_t z_k^m(0)=0.
\end{aligned}
\right.
\end{equation*}

Setting
\begin{eqnarray*}
E_{v_k^m}^L(t) := \int_\Omega \left( |\partial_t v_k^m(x,t)|^2 + |\nabla v_k^m(x,t)|^2\right)\,dx,
\end{eqnarray*}
the linear part associated with energy $E_{v_k^m}(t)$, then we can write
\begin{eqnarray}\label{energy decomposition}
E_{v_k^m}(t)= E_{v_k^m}^L(t) + \frac{1}{\alpha_m^2} \int_\Omega F_k(u_k((x,t)))\,dx.
\end{eqnarray}

In the sequel, let us estimates the nonlinear term of the RHS of (\ref{energy decomposition}) in terms of $E_{v^m}^L(t)$.

\medskip
\noindent{Estimate for $I_1:=\frac{1}{\alpha_m^2} \int_\Omega F_k(u_k(x,t))\,dx$.}
\medskip

Taking (\ref{ass on F}) into account, one has
\begin{eqnarray*}
|I_1|&\leq&  \frac{1}{\alpha_m^2} \int_\Omega \left[ |u_k^m|^2 + |u_k^m|^{p+1}\right]\,dx\\
&=&  \frac{1}{\alpha_m^2} \left[ ||u_k^m(t)||_{L^2(\Omega)}^2 +  ||u_k^m(t)||_{L^{p+1}(\Omega)}^{p+1}\right].
\end{eqnarray*}

If $p=1$, it follows, in  view of (\ref{def v_k}), that
\begin{eqnarray*}
|I_1|\leq \frac{2}{\alpha_m^2}||u^m(t)||_{L^2(\Omega)}^2 = 2||v_k^m(t)||_{L^2(\Omega)}^2 \lesssim E_{v_k^m}^L(t).
\end{eqnarray*}

Now, if $p>1$ then $p+1>2$ and, therefore $p+1=2+\varepsilon$ for some $\varepsilon>0$, Thus, having in mind that the map $t \mapsto E_{v^m}$ is non increasing and $E_{v^m}(0)=1$, we infer
\begin{eqnarray*}
|I_1| &\leq& \left[ ||v_k^m(t)||_{L^2(\Omega)}^2 + \alpha_m^{p-1} ||v_k^m(t)||_{L^2(\Omega)}^{p+1}\right]\\
&=& \left[ ||v_k^m(t)||_{L^2(\Omega)}^2 + \alpha_m^{p-1} ||v_k^m(t)||_{L^2(\Omega)}^{2+\varepsilon}\right]\\
&\lesssim& E_{v_k^m}^L(t) + \alpha_m^{p-1} [E_{v_k^m}^L(t)] [E_{v_k^m}(0)]^{\frac{\varepsilon}{2}}\\
&\lesssim& [1 + \alpha_m^{p-1}]E_{v_k^m}^L(t).
\end{eqnarray*}

In any case, we deduce
\begin{eqnarray}\label{bound 1}
|I_1| \lesssim E_{v_k^m}^L(t).
\end{eqnarray}

So, combining (\ref{energy decomposition}) and (\ref{bound 1}) we obtain
\begin{eqnarray}\label{bound 3}
E_{v_k^m}(t) \lesssim E_{v_k^m}^L(t),~\hbox{ for all }t\in[0,T].
\end{eqnarray}

Now, employing the observability given in (\ref{Linear Obs Ineq}) and having  in mind that $E_{v_k^m}^L(0)\equiv E_{y_k^m}(0)$, we deduce from (\ref{bound 3}) that
\begin{eqnarray}\label{bound 4}
E_{v_k^m}(0) &\lesssim& E_{v_k^m}^L(0)= E_{y_k^m}(0)\leq c \int_0^T \int_\omega |\partial_t y_k^m(x,t)|^2\,dxdt.
\end{eqnarray}

From (\ref{bound 4}), observing that $a(x) \geq a_0>0$ in $\omega$ and since $v_k^m = y_k^m + z_k^m$, we obtain
\begin{eqnarray}\label{bound 5}
E_{v_k^m}(0)  &\lesssim& \int_0^T\int_\Omega a(x) |\partial_t v_k^m(x,t)|^2\,dxdt + \int_0^T\int_\Omega |\partial_t z_k^m(x,t)|^2\,dxdt.
\end{eqnarray}

On the other hand, using the well-known result which establishes that the map $\{z_0, z_1, f\}\mapsto \{z, \partial_t z\}\in L^\infty(0,T; H^1_{\partial \Omega_D}(\Omega))\times L^\infty(0,T;L^2(\Omega))$ associating the initial data $\{z_0, z_1, f\}\in H^1_{\partial \Omega_D}(\Omega) \times L^2(\Omega) \times L^1(0,T;L^2(\Omega))$ to the unique solution to the linear problem
	\begin{equation}\label{aux prob2}
	\left\{
	\begin{aligned}
	& \partial_t^2 z  - \Delta z = f~\hbox{ in }\Omega\times (0,T)\\
    &z=0 ~\hbox{ on } \partial \Omega_D \times (0,T),\\
    &\partial_\nu z=0 ~\hbox{ on } \partial \Omega_N \times (0,T),\\
	&z(0)=z_0,~\partial_t z(0)=z_1
	\end{aligned}
	\right.
	\end{equation}
	is linear and continuous; we obtain, from (\ref{bound 5}), and, in particular, considering $z_0=z_1=0$ and $f:=- \frac{1}{\alpha_m} f_k(u_k^m) -a(x) \partial_t u_k^m$, that
\begin{eqnarray}\label{bound 6}
E_{v_k^m}(0)  &\lesssim& \int_0^T\int_\Omega a(x) |\partial_t v_k^m(x,t)|^2\,dxdt
+  \frac{1}{\alpha_m^2} \int_0^T \int_\Omega|f_k(u_k^m)|^2  \,dxdt.
\end{eqnarray}

 Thus, from (\ref{damping conv}), (\ref{crucial bound}) and (\ref{bound 6}) we deduce that $E_{v_k^m}(0) \rightarrow 0$ as $m\rightarrow +\infty$ as desire to prove in (\ref{main goal}).

\end{proof}

\medskip

In what follows, we are going to conclude the exponential stability to the problem (\ref{eq:*}).

Thanks to inequality (\ref{obs ineq}), the auxiliary problem (\ref{eq:AP}) satisfies the following observability inequality:
\begin{equation}\label{trunc observ ineq}
E_{u_k}(0) \leq C\,\int_0^{T}\int_{\Omega} a(x) |\partial_t u_k|^2\,dx\,dt,\hbox{ for all } T\geq T_0, \hbox{ and } k\in \mathbb{N},~ k \geq k_0,
\end{equation}
where $C$ is a positive constant which does not depend on $k\in \mathbb{\mathbb{N}}$.

Passing to the limit as $k\rightarrow +\infty$ and observing convergences  (\ref{Cauchy conv}), (\ref{damping conv'}) and (\ref{main strong conv}), the above inequality yields  the observability inequality associated to the original problem (\ref{eq:*}), that is,
\begin{equation}
\label{final obs_inequality}
E_{u}(0) \leq C\,\int_0^{T}\int_{\Omega} a(x) | \partial_t u|^2 \,dx\,dt, \hbox{ for all } T\geq T_0.
\end{equation}	 	 	

On the other hand, passing to the limit as $k\rightarrow +\infty$ and considering the same convergences  (\ref{Cauchy conv}), (\ref{damping conv'}) and (\ref{main strong conv}), identity (\ref{est2}) yields the identity associated to the original problem (\ref{eq:*}), namely,
\begin{eqnarray}\label{final energy ident}\quad
E_u(t_2) - E_u (t_1) + \int_{t_1}^{t_2} \int_\Omega  a(x) |\partial_t u|^2 \,dx\,dt=0,~\hbox{ for all }0\leq t_1 < t_2<+\infty.
\end{eqnarray}

Gathering together (\ref{final obs_inequality}), (\ref{final energy ident}), and since the map $t \mapsto E_{u}(t)$ is a non-increasing function, we obtain
\begin{equation}
\begin{aligned}
E_{u}(T_0)&\leq C\,\int_0^{T_0}\int_{\Omega}\left(a(x) |\partial_t u|^2\right)\,dx\,dt\\&=C\,\left(E_{u}(0)-E_{u}(T_0)\right),
\end{aligned}
\end{equation}
that is,
\begin{eqnarray}
\label{inequality}E_{u}(T_0)&\leq&\left(\frac{C}{1+C}\right)\,E_{u}(0).
\end{eqnarray}

Repeating the same steps for $m T_0$, $m\in \mathbb{N}, m\geq 1$, we deduce
$$E_{u}(m T_0)\leq\frac{1}{(1+\hat{C})^m}E_{u}(0),$$
where $\hat{C}=C^{-1}$. Consider $t\geq T_0$ and $t=m T_0+r,$ $0\leq r<T_0$. Thus,
$$E_{u}(t)\leq E_{u}(t-r)=E_{u}(m T_0)\leq \frac{1}{(1+\hat{C})^m}\,E_{u}(0)=\frac{1}{(1+\hat{C})^{\frac{t-r}{T_0}}}E_{u}(0).$$
	 	
Defining $\displaystyle C:=\textrm{e}^{\frac{r}{T_0}\ln(1+\hat{C})}$ and $\lambda_0:=\frac{\ln(1+\hat{C})}{T_0}>0,$ we obtain
\begin{equation}\label{exp'}
E_{u}(t)\leq C \,\textrm{e}^{- \lambda_0 t}E_{u}(0) \hbox{ for all } t\geq
T_0,
\end{equation}
which proves the exponential decay  to problem (\ref{eq:*}) and we prove the following result.

\begin{Theorem}\label{theo 3}
Under the assumptions of Theorem \ref{theo 2} and Assumptions \ref{ass 1} and \ref{UCP}  there exist positive constants $C$ and $\gamma$ such that the following exponential decay holds
\begin{equation}\label{exp}
	 	 	E_{u}(t)\leq C \,\textrm{e}^{- \lambda_0 t}E_{u}(0), \hbox{ for all } t\geq T_0.
\end{equation}
for every solution to problem (\ref{eq:*}), provided that the initial data are taken in bounded sets of the phase-space $\mathcal{H}:= H^1_{\partial \Omega_D}(\Omega) \times L^2(\Omega)$.
\end{Theorem}


\begin{thebibliography}{}

\end{thebibliography}


\begin{thebibliography}{99}
\bibliographystyle{plain}

  %%%%%%%%%%%%%%%%%%%%%%%%%%%%%%%%%%%%%%%%%%%
 \bibitem{alabau:05}
F.~Alabau-Boussouira.
\newblock Convexity and weighted integral inequalities for energy decay rates
  of nonlinear dissipative hyperbolic systems.
\newblock Appl. Math. Optim., \textbf{51}(1):61--105, 2005.

\bibitem{ademirmaria} M. Astudillo, M. M. Cavalcanti, V, N. Domingos Cavalcanti, R. Fukuoka, A. B. Pampu.
Uniform Decay Rates estimates for the semilinear wave equation in inhomogeneous media
with locally distributed nonlinear damping, {\em Nonlinearity}, 31. (2018),
pp. 4031-4064.


\bibitem{Bardos}%
  C. Bardos, G. Lebeau, and J. Rauch,
  Sharp sufficient conditions for the observation, control, and stabilization of waves from the boundary,
   SIAM J. Control Optim. \textbf{30}, 1024-1065 (1992).


\bibitem{Brezis}
H. Brezis, Functional Analysis, Sobolev Spaces and Partial Differential Equations. Universitext.
Springer, New York (2011).


\bibitem{Burq-Gerard-CR}
  N. Burq, and P. G\'erard,
  Condition n\'ecessaire et suffisante pour la contr\^olabilit\'e exacte des ondes. (French) [A necessary and sufficient condition for the exact controllability of the wave equation],
  C. R. Acad. Sci. Paris S\'er. I Math. \textbf{325}, 749-752 (1997).

%\bibitem{Cavalcanti1} M.M. Cavalcanti, V.N.D. Cavalcanti, V. Komornik. Introdu\c c\~ao a An\'alise Funcional. Maring\'a: Eduem, 2011.

\bibitem{Cavalcanti2}
  M.\,M. Cavalcanti, V.\,N. Domingos Cavalcanti, R. Fukuoka, and J.\,A. Soriano,
  Asymptotic stability of the wave equation on compact manifolds and locally distributed damping: a sharp result,
  Arch. Ration. Mech. Anal. \textbf{197}, 925-964 (2010).

\bibitem{Cavalcanti3}
  M.\,M. Cavalcanti, V.\,N. Domingos Cavalcanti, R. Fukuoka and J. A. Soriano,
  Asymptotic stability of the wave equation on compact surfaces and locally distributed damping-a sharp result,
  Trans. Amer. Math. Soc. \textbf{361}, 4561-4580 (2009).

 \bibitem{Cornilleau} P. Cornilleau and L. Robbiano,  Carleman estimates for the Zaremba boundary condition and stabilization of waves. Amer. J. Math. 136 (2014), no.  \textbf{2}, 393-444.

\bibitem{Dehman0}
  B. Dehman,
  Stabilisation pour l'\'equation des ondes semilin\'eaire,
  Asymptotic Anal. \textbf{27}, 171-181 (2001).

\bibitem{Dehman2}
  B. Dehman, G. Lebeau, and E. Zuazua,
  Stabilization and control for the subcritical semilinear wave equation,
  Anna. Sci. Ec. Norm. Super. \textbf{36}, 525-551 (2003).

\bibitem{DZZ}
 T. Duyckaerts, X. Zhang and E. Zuazua, On the optimality of the observability inequalities for parabolic and hyperbolic systems with potentials, Ann. Inst. H. Poincaré Anal. Non Linéaire \textbf{25}, 1-41 (2008).

\bibitem{Gerad}
  P. G\'erard,
  Microlocal defect measures,
  Comm. Partial Differential Equations \textbf{16}, 1761-1794 (1991).

\bibitem{Gerard-JFA}
  P. G\'erard,
  Oscillations and concentration effects in semilinear dispersive wave equations,
  J. Funct. Anal. \textbf{141}, 60-98 (1996).

 \bibitem{Haraux} A. Haraux,  Une remarque sur la stabilisation de certains systèmes du deuxième ordre en temps. (French) [A remark on the stabilization of certain systems of second order in time] Portugal. Math. 46 (1989), no. \textbf{3}, 245-258.

\bibitem{Haraux2} A. Haraux,  Stabilization of trajectories for some weakly damped hyperbolic equations,
  J. Differential Equations \textbf{59}, 145-154 (1985).

\bibitem{Joly}
  R. Joly, and C. Laurent,
  Stabilization for the semilinear wave equation with geometric control,
  Analysis $\&$ PDE \textbf{6}, 1089-1119 (2013).

\bibitem{Huang} F. L. Huang, Characteristic conditions for exponential stability of linear dynamical
systems in Hilbert spaces, Ann. Differential Equations \textbf{1} (1985) 43–56.

%\bibitem{Tataru}
%   H. Koch, and D. Tataru,
%   Dispersive estimates for principally normal pseudodifferential operators,
%   Commun. Pure Appl. Math. \textbf{58}, 217-284 (2005).

  \bibitem{Lasiecka-Tataru} {I. Lasiecka and D. Tataru}, Uniform boundary
stabilization of semilinear wave equation with nonlinear boundary
damping, Differential and integral Equations,  \textbf{6} (1993), 507-533.


\bibitem{Laurent1}
   C. Laurent,
   On stabilization and control for the critical Klein Gordon equation on 3-D
compact manifolds,
   Journal of Functional Analysis \textbf{260}, 1304-1368 (2011).

\bibitem{Lebeau}
   G. Lebeau,
   Equations des ondes amorties,
   Algebraic Geometric Methods in Maths. Physics  73-109 (1996).

\bibitem{Lions1} J.L. Lions, Quelques Meth\'odes de Resolution des Probl\'ems aux limites Non Line\'eires, 1969.

\bibitem{Lions-Magenes} J. L. Lions and E. Magenes. Probl\'emes aux Limites non Homog\`enes,
Aplications, Dunod, Paris, 1968.


  \bibitem{martinez:99:RMC}
P.~Martinez.
\newblock A new method to obtain decay rate estimates for dissipative systems
  with localized damping.
\newblock Rev. Mat. Complut., \textbf{12}(1):251--283, 1999.


\bibitem {Miller}
   L. Miller.
   Escape function conditions for the observation, control,
and stabilization of the wave equation,
   SIAM J. Control Optim. \textbf{41}, 1554-1566 (2002).

\bibitem{Nakao}
   M. Nakao,
   Energy decay for the linear and semilinear wave equations in exterior domains with some
localized dissipations,
   Math. Z. \textbf{4} 781-797 (2001).

\bibitem{Pazy}
   A. Pazy,
   Semigroups of Linear Operators and Applications to Partial Differential Equations,
   Applied Mathematical Sciences 44 (Springer-Verlag, New York, 1983).

\bibitem{RT}
   J. Rauch and M. Taylor,
   Decay of solutions to nondissipative hyperbolic systems on compact manifolds,
   Comm. Pure Appl. Math. \textbf{28} 501-523 (1975).

\bibitem{Ruiz}
   A. Ruiz,
   Unique Continuation for Weak Solutions of the Wave Equation  plus a Potential,
   J. Math. Pures. Appl. \textbf{71} 455-467 (1992).


\bibitem{Simon}
   J. Simon,
   Compact Sets in the space $L^p(0, T; B)$,
   Ann. Mat. Pura Appl. \textbf{146} 65-96 (1987).

   \bibitem{Strauss} W. A. Strauss, On weak solutions of semilinear hyperbolic equations, Anais da Academis Brasileira de Ci\^encias, \textbf{71}, 1972, 645-651.

 \bibitem{Tebeou-localized} L. Tebou, Stabilization of the wave equation with localized nonlinear damping. Journal of Differential Equations 1998; \textbf{145}:502-524.

\bibitem{Tebou}
   L. Tebou,
   Stabilization of some elastodynamic systems with localized Kelvin-Voigt damping,
   Discrete and Continuous Dynamical Systems, \textbf{36} 7117-7136 (2016).

\bibitem{toundykov:07}
D.~Toundykov.
\newblock Optimal decay rates for solutions of a nonlinear wave equation with
  localized nonlinear dissipation of unrestricted growth and critical exponent
  source terms under mixed boundary conditions.
\newblock  Nonlinear Anal., \textbf{67}(2):512--544, 2007.

\bibitem{Triggiani} R. Triggiani and P. F. Yao, Carleman estimates with no lower-order terms for general Riemann wave equations. Global uniqueness and observability in one shot. Special issue dedicated to the memory of Jacques-Louis Lions. Appl. Math. Optim. \textbf{46} (2002), no. 2-3, 331–375.


\bibitem{Zuazua}
   E. Zuazua,
   Exponential decay for semilinear wave equations with localized damping,
   Comm. Partial Differential Equations,  \textbf{15} 205-235 (1990).

\end{thebibliography}
\end{document}